\documentclass[11pt,letterpaper]{article}
\usepackage[latin1]{inputenc}
\usepackage{amsthm,amsmath,color}
\usepackage[numbers]{natbib}
\usepackage{amsfonts,enumerate,fullpage}
\usepackage{amssymb}
\usepackage{graphicx}
\usepackage{url}
\usepackage[ruled,vlined,linesnumbered]{algorithm2e}

\newtheorem{theorem}{Theorem}[section]
\newtheorem{proposition}[theorem]{Proposition}

\newtheorem{lemma}[theorem]{Lemma}
\newtheorem{corollary}[theorem]{Corollary}

\newtheorem{conjecture}{Conjecture}

\def\ZZ{\mathbb{Z}}

\def\bfd{\boldsymbol{d}}
\def\bft{\boldsymbol{t}}
\def\bfB{\boldsymbol{B}}
\def\GG{\mathbb{G}}
\def\HH{\mathbb{H}}

\title{On the degree sequences of dual graphs on surfaces}

\author{Endre Boros\thanks {MSIS \& RUTCOR, Business School, Rutgers University;
endre.boros@rutgers.edu}
\and
Vladimir Gurvich\thanks {National Research University Higher School of Economics (HSE), Moscow, Russia;
vgurvich@hse.ru, vladimir.gurvich@gmail.com}
\and
Martin Milani\v{c}\thanks{FAMNIT and IAM, University of Primorska, Koper, Slovenia; martin.milanic@upr.si}
\and
Jernej Vi\v{c}i\v{c}\thanks{FAMNIT and IAM, University of Primorska, Koper, Slovenia and
Research Centre of the Slovenian Academy of Sciences and Arts, The Fran Ramov\v{s} Institute;
jernej.vicic@upr.si}}

\begin{document}
\maketitle

\begin{sloppypar}
\begin{abstract}
Given two graphs $G$ and $G^*$ with a one-to-one correspondence between their edges, when do $G$ and $G^*$ form a  pair of dual graphs realizing the vertices and countries of a map embedded in a surface? A criterion was obtained by Jack Edmonds in 1965. Furthermore, let $\bfd=(d_1,\ldots,d_n)$ and $\bft=(t_1,\ldots,t_m)$  be their degree sequences. Then, clearly, $\sum_{i=1}^n d_i = \sum_{j=1}^m t_j = 2\ell$, where $\ell$ is the number of edges in each of the two graphs, and $\chi = n - \ell + m$ is the Euler characteristic of the surface. Which sequences $\bfd$ and $\bft$  satisfying these conditions still cannot be realized as the degree sequences? We make use of Edmonds' criterion to obtain several infinite series of exceptions for the sphere, $\chi = 2$, and projective plane, $\chi = 1$.
We conjecture that there exist no exceptions for $\chi \leq 0$.


\end{abstract}
\end{sloppypar}

\section{Introduction}

We consider embeddings (drawings) of a graph $G$ in a surface $S$ where vertices are mapped to distinct points on the surface and edges are drawn as curves connecting the images of their endpoints so that different edges do not meet, except in their common endpoints (see, e.g.,~\cite{MR1844449}). Removing the image of the embedding from the surface results in one or more connected components, which are called {\it countries}. Furthermore, we assume that each country is homeomorphic to an open disk and call such embeddings {\it maps}. It follows from this assumption that graph $G$ must be connected. Introduce a graph $G^*$ dual to $G$  realizing the neighbor relations among countries. The graphs $G$ and $G^*$ have the same set of edges. More precisely, there is a natural one-to-one correspondence between their edge sets. An arbitrary pair of graphs with common set of edges is called a {\it plan}. Every map induces a plan. A plan is called {\it geographic} if it is induced by a map.

In 1965, Edmonds obtained necessary and sufficient conditions for a plan to be geographic~\cite{Edmonds}. This result was rediscovered, with essentially the same proof, in 1989 by~\citet{GurvichShabat}, see also~\cite{Gurvich-RRR} for more details. The proof is constructive and allows to construct all maps inducing the given geographic plan. Edmonds' result and techniques have found
several applications to graphs on surfaces~\cite{MR2482947,MR3184644,MR898434,MR1179495,MR979055}.
The result was also generalized to partial duality~\cite{MR2507944,MR3071855,MR2817160}.

Given a geographic plan, let $\ell$ be the number of edges in each of the two graphs.
Furthermore, let $\bfd=(d_1,\ldots,d_n)$ and $\bft=(t_1,\ldots,t_m)$  be their degree sequences. Then, clearly,
\begin{equation}\label{e:degree-sum}
\sum_{i=1}^n d_i = \sum_{j=1}^m t_j = 2\ell\,.
\end{equation}
The Euler characteristic of the surface in which the two graphs are embedded is given by the formula $\chi = n - \ell + m$ (see Section~\ref{sec:preliminaries} for details).

The degree sequences of connected graphs were characterized in~\cite{MR0148049}. However, characterizing degree sequences of geographic plans seems to be a much more difficult problem. This problem is the main motivation for our study. In particular, we focus on the following question. Which pairs of sequences $\bfd$ and $\bft$ of non-negative integers satisfying condition~\eqref{e:degree-sum} still cannot be realized as the degree sequences of a geographic plan? We make use of Edmonds' criterion to obtain several infinite series of exceptions for the sphere, $\chi = 2$, and for the projective plane, $\chi = 1$. We conjecture that there exist no exceptions for $\chi \leq 0$.

\section{Preliminaries}\label{sec:preliminaries}

\noindent{\bf Graphs.} We only consider finite undirected graphs, with loops and
multiple edges allowed. We need the following standard concepts.

A cyclic sequence of alternating vertices and edges in which any
consecutive edge and vertex are incident is called a {\it closed walk}.
A closed walk in which all the edges are different is called a
{\it closed trail}. A closed trail in which all vertices are different
is called a {\it cycle}. The {\it degree} of a vertex is defined as the number of edges incident to it, where loops are considered with multiplicity $2$. The degree of vertex  $v$  in a graph  $G$ is denoted by $d_G(v)$. A closed trail containing all the edges of a graph is called an {\it Eulerian trail} and a graph having an Eulerian trail is called {\it Eulerian}. Clearly, an Eulerian graph without isolated vertices (of degree $0$) must be connected. Criteria for a connected graph to be Eulerian are given by the following theorem (see, for instance,~\cite[Theorem 7.1]{Harary}).

\begin{theorem}[Euler's ``K\"onigsberg Bridge'' theorem~\cite{euler1736solutio}]\label{thm:Eulerian}
For a connected graph $G$ the following three properties are equivalent:
\begin{enumerate}[a)]
  \item $G$ is Eulerian.
  \item There exist a collection of cycles in $G$ such that their
edge sets represent a partition of the edge set of the graph.
  \item All vertices have even degrees.
\end{enumerate}
\end{theorem}

\medskip
\begin{sloppypar}
\noindent{\bf Surfaces.} The topological classification of surfaces
(that is, two-dimensional compact manifolds without boundary)
is well known. Each surface is either {\it orientable} (in which case it is homeomorphic to $S_p$, a sphere with $p\geq  0$ handles) or {\it non-orientable} (in which case it is homeomorphic to $C_q$, a sphere with $q\geq  1$ holes glued by M\"obius strips); see, for instance,~\citet{LandoZvonkin,MR1844449}. For example, $S_0$  is a sphere, $S_1$ is a torus, $C_1$ is a projective plane, and $C_2$ is a Klein bottle.
\end{sloppypar}

\medskip
\noindent{\bf Maps.} Let $S$ be a surface and  $G = (V,E)$ be
a graph where $V =\{v_1,\ldots, v_n\}$ is the set of vertices,
$E=\{e_1,\ldots, e_\ell\}$ is the set of edges.
Furthermore, let $\phi$ be an embedding of $G$ in $S$ such that
the edges do not have intersections on the surface apart from their
common vertices in the graph, and also they have no self-intersections
apart from the vertices of loops. Let us cut surface $S$ along the edges
of graph $G$; in other words, partition the difference $S - \phi(G)$
into connected components (countries). Recall that every country must be
homeomorphic to an open disk and therefore graph $G$ must be connected.
Denote the set of countries by $F=\{f_1,\ldots, f_m\}$.

A triple $M=(S,G,\phi)$ satisfying the conditions stated above will
be called a {\it map}. For a map $M$ as above, we write $V(M) = V$, $E(M) = E$,
and $F(M) = F$. Standardly, two maps  $M'= (S',G',\phi')$ and
$M'' = (S'', G'', \phi'')$ are called {\it isomorphic} if there
exists a homeomorphism $g: S' \to S''$ carrying $\phi'(G')$  into
$\phi''(G'')$.

\medskip
\begin{sloppypar}
\noindent{\bf The Euler characteristic.}
It is well known that for every map $M$ on a given surface $S$ the number $|V(M)|-|E(M)|+|F(M)|$ takes the same value, called the {\it Euler characteristic} of the surface and denoted by $\chi(S)$, that is,
\begin{equation}\label{e:chi}
\chi(S)=|V(M)|-|E(M)|+|F(M)|\,.
\end{equation}
For the surfaces $S_p$ and $C_q$ we have
\begin{equation}\label{e:chi-surfaces}
\chi (S_p)=2-2p,\ p=0,1,2,\ldots\,;\ \ \chi (C_q)=2-q,\ q=1,2,\ldots\,.
\end{equation}
Therefore we always have $\chi(S)\le 2$ , and  $\chi(S) = 2$ if and only if $S$ is a sphere.
For a given value of $\chi\le 2$ there are at most two surfaces
with Euler characteristic $\chi$, namely, $C_q$ with $q=2- \chi $ and  $S_p$ with $p = 1
-\chi/2$,
where the second one exists only for even $\chi$.
\end{sloppypar}

\medskip
\noindent{\bf Word representations of surfaces and maps.} The classical
combinatorial approach to surfaces will be considered here briefly;
for more details, see~\cite{LandoZvonkin,MR1844449}.
For a positive integer $n$, an {\it $n$-gon} is a disk in the plane with $n$ distinct {\it points} on its boundary, which partition the boundary into $n$ {\it sides}. Any $n$-gon is also called a {\it polygon}. Fix a number of pairwise disjoint polygons in the plane. Each one can have any number of sides (including 2 and 1). We assume that the sum of all these numbers is even and that the set of all the sides is partitioned into pairs. In each polygon, fix an arbitrary direction for each side and denote the sides directed clockwise by
$a$, $b$, $c, \ldots$ and counterclockwise by $\overline{a}$, $\overline{b}$, $\overline{c}, \ldots$
Two sides in any pair are denoted by the same letter. Thus in the set of all the polygons each letter occurs twice. We also assume that the set of all polygons is minimal with respect to this property. In other words, for any proper subset of the set of all polygons there exists a letter that only occurs once. (Otherwise, we would get more than one surface.)

\medskip
If we glue all the pairs of sides denoted by the same letter in accordance with their directions, we obtain a surface $S$. Note that along with the surface, we also obtain a map $(S,G,\phi)$. Each edge of $G$ corresponds to one of the letters and the corresponding pair of polygonal sides; the vertices of the graph correspond to the endpoints of the polygonal sides, taking into account also the side identifications.
Furthermore, every polygon becomes a country in this map. The above representation of the surface $S$ and the corresponding map will be called a {\it word representation}. We emphasize that this is a labeled representation, in the sense that the edges of the resulting graph $G$ will be labeled. See Fig.~\ref{fig:example1} for an example.

\medskip
\begin{figure}[h!]
	\begin{center}
		\includegraphics[width=\textwidth]{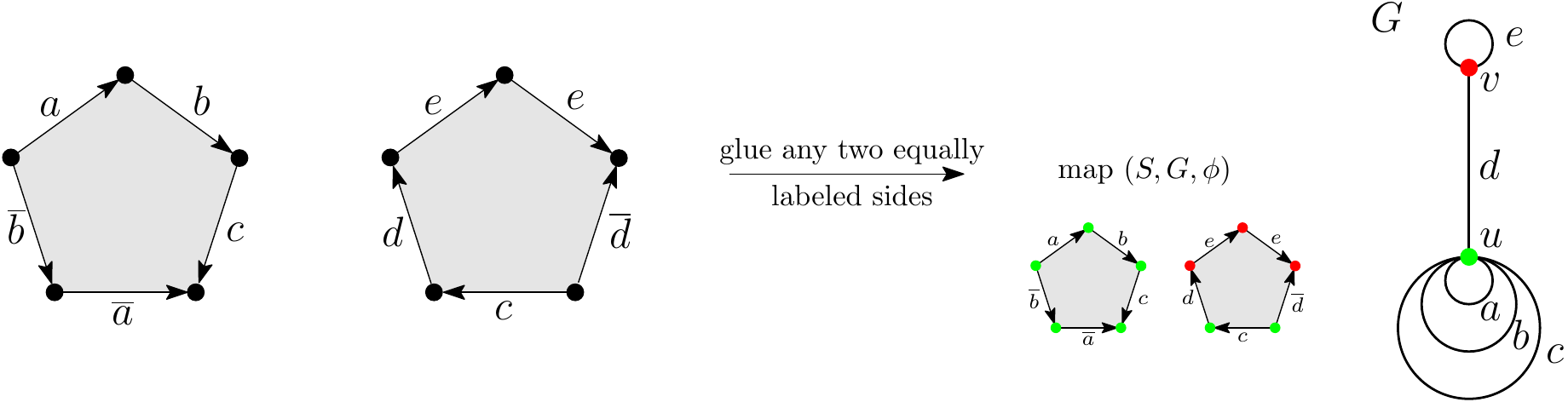}
	\end{center}
	\caption{An example of a word representation of a map. There are two $5$-gons and the corresponding two words are, for example: $(a\ b\ c\ \overline{a}\ \overline{b})$, $(c\ {d}\  e\ e\ \overline{d})$. After the identification of the sides, we obtain a map $(S,G,\phi)$ with $2$ vertices \hbox{($u$ and $v$)}, $5$ edges, and $2$ countries. The Euler characteristic of surface $S$ is $-1$, hence $S$ is the non-orientable surface $C_3$.}\label{fig:example1}
\end{figure}

\medskip
Note that in the above construction, each labeled polygon with $s$ sides gives rise to an Eulerian subtour (that is, a closed trail) of length $s$ in the multigraph $\GG$ obtained from $G$ by duplicating each edge. Furthermore, these Eulerain subtours can be chosen so that they are pairwise edge-disjoint, in which case each edge of $\GG$ appears exactly once as an edge of one of these subtours. In other words, these subtours decompose $\GG$. In the example given by Fig.~\ref{fig:example1}, the labeled polygons give rise to the following two closed trails, respectively: $(u,a',u,b',u,c',u,a'',u,b'',u)$ and $(u,d',v,e',v,e'',v,d'',u,c'',u)$, see Fig.~\ref{fig:example2}.

\begin{figure}[h!]
	\begin{center}
		\includegraphics[width=0.7\textwidth]{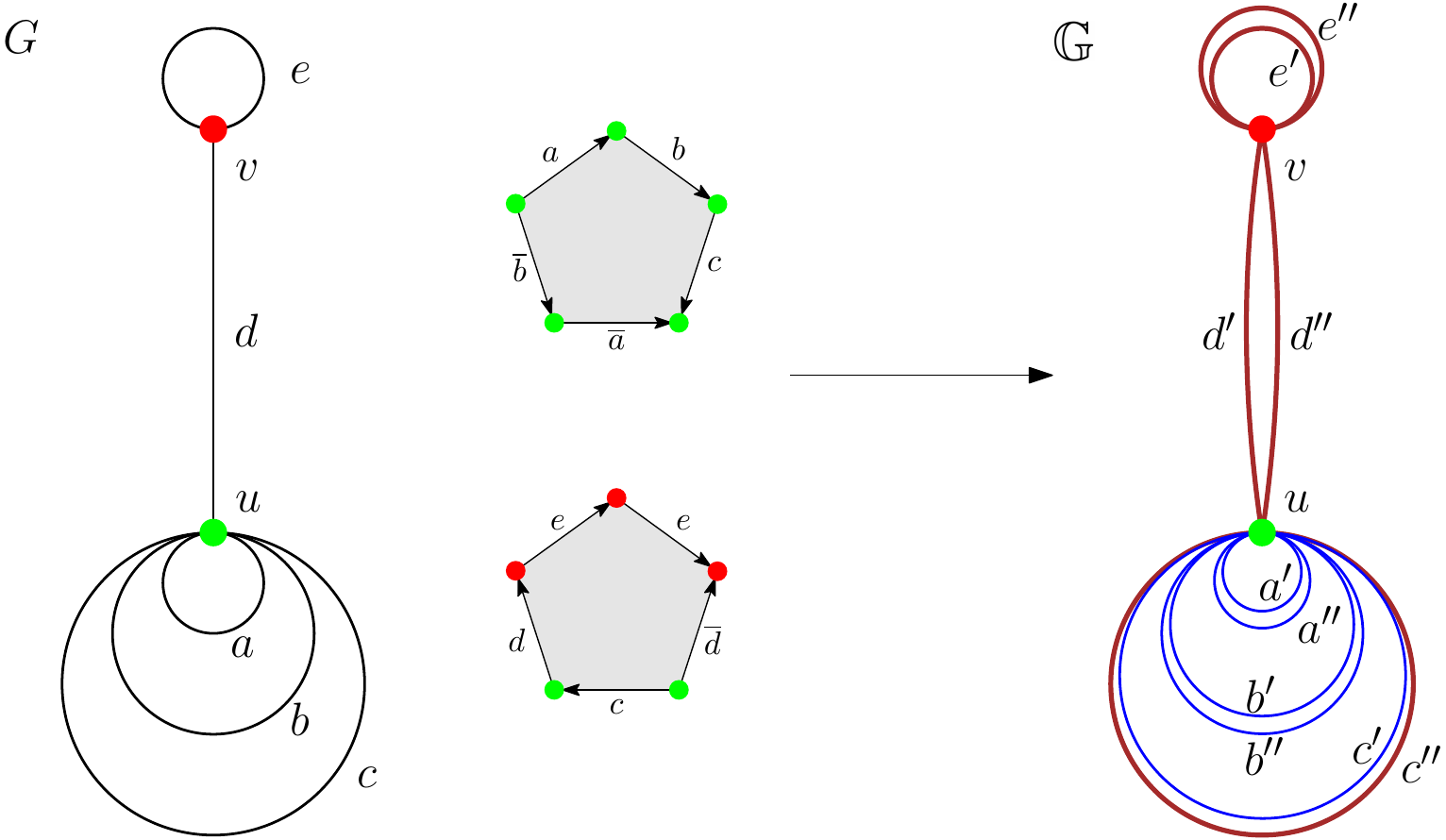}
	\end{center}
	\caption{An example of a decomposition of $\GG$ into Eulerian subtours.}\label{fig:example2}
\end{figure}

\medskip
We are interested in reversing the above procedure. Thus, let $E = \{a,b,c,\ldots\}$ be the set of $m$ edges of a graph $G$. Let us orient each of these edges arbitrarily. Let us also consider a family of Eulerian subtours decomposing $\GG$ and a family of pairwise disjoint polygons in the plane, each corresponding to one of the subtours in $\GG$. In particular, the polygons have exactly $2m$ sides in total. Each of the subtours gives rise to a labeling of the sides of the corresponding polygon.
We label each of the sides by either $e$ or $\overline{e}$ where $e$ is a directed edge from $E$ depending on the direction in which the edge is traversed in the subtour.
After this labeling procedure is finished, each edge labels exactly two sides.
The labeling procedure also defines an orientation of the sides of the polygons. Given a directed edge $e = (x,y)$, a side label of the form $e$ means that the side is oriented clockwise from $x$ to $y$, while a label of the form $\overline{e}$ means that the side is oriented counterclockwise. If we glue all the pairs of sides denoted by the same letter in accordance with their directions, we obtain a surface $S$.\footnote{Strictly speaking, this is true under the assumption that the set of polygons is minimal for the property that each of the edges labeling their sides appears exactly twice -- if this is not the case, then we get two (or more) surfaces.} Note that along with the surface, we also obtain a map $(S,G,\phi)$. Each edge of $G$ corresponds to one of the letters and the corresponding pair of polygonal sides; the vertices of the graph correspond to the endpoints of the polygonal sides, taking into account also the side identifications. Furthermore, every polygon becomes a country in this map. The above representation of the surface $S$ and the corresponding map will be called a {\it word representation}. We emphasize that this is a labeled representation, in the sense that the edges of the resulting graph $G$ are labeled. Note that this process is far from unique; see Fig.~\ref{fig:example} and~\ref{fig:example3} for examples.

\medskip
\begin{figure}[h!]
\begin{center}
\includegraphics[width=\textwidth]{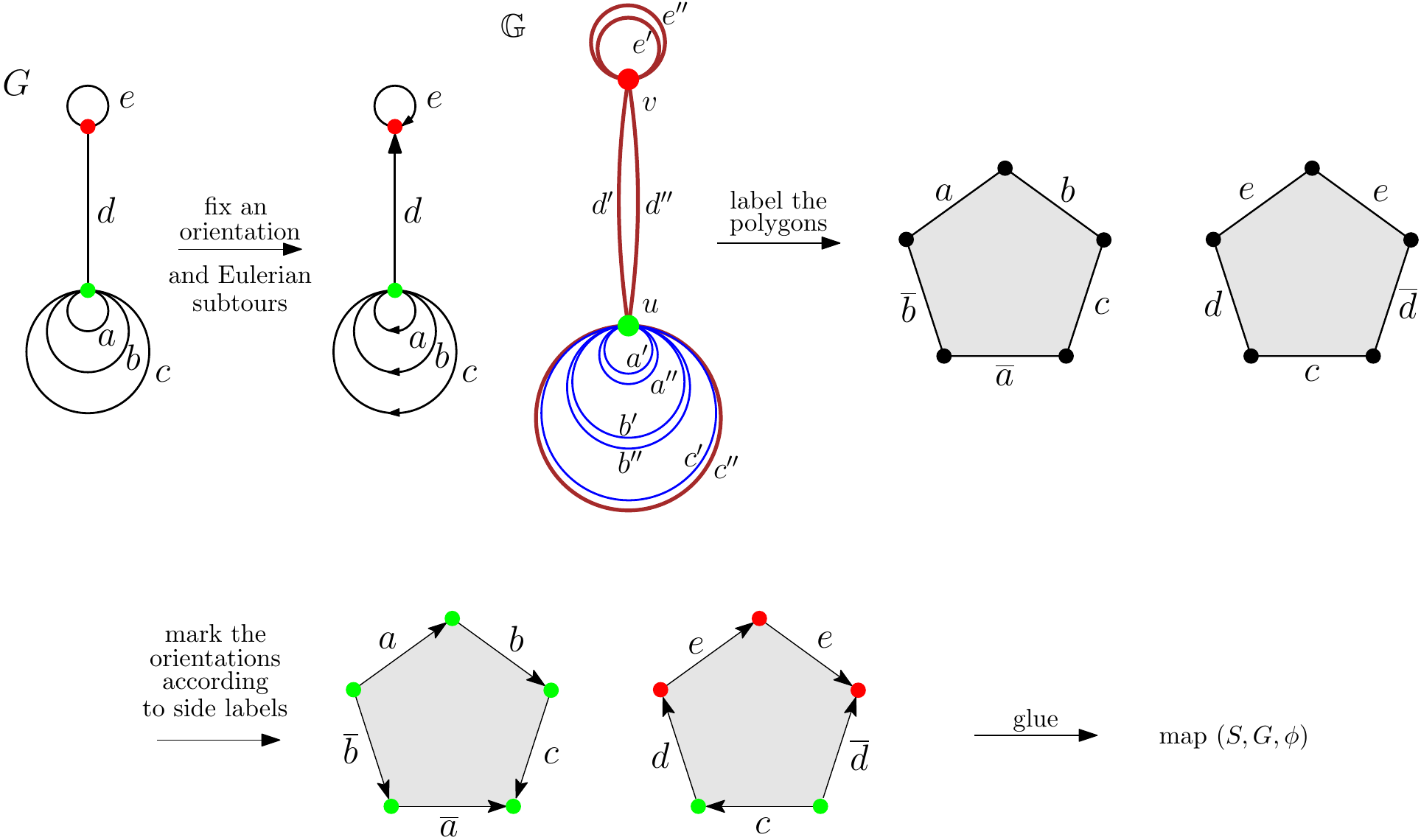}
\end{center}
\caption{An example of reversing the construction, using the Eulerian subtours $(u,a',u,b',u,c',u,a'',u,b'',u)$ and $(u,d',v,e',v,e'',v,d'',u,c'',u)$.}\label{fig:example}
\end{figure}

\medskip
\begin{figure}[h!]
	\begin{center}
		\includegraphics[width=\textwidth]{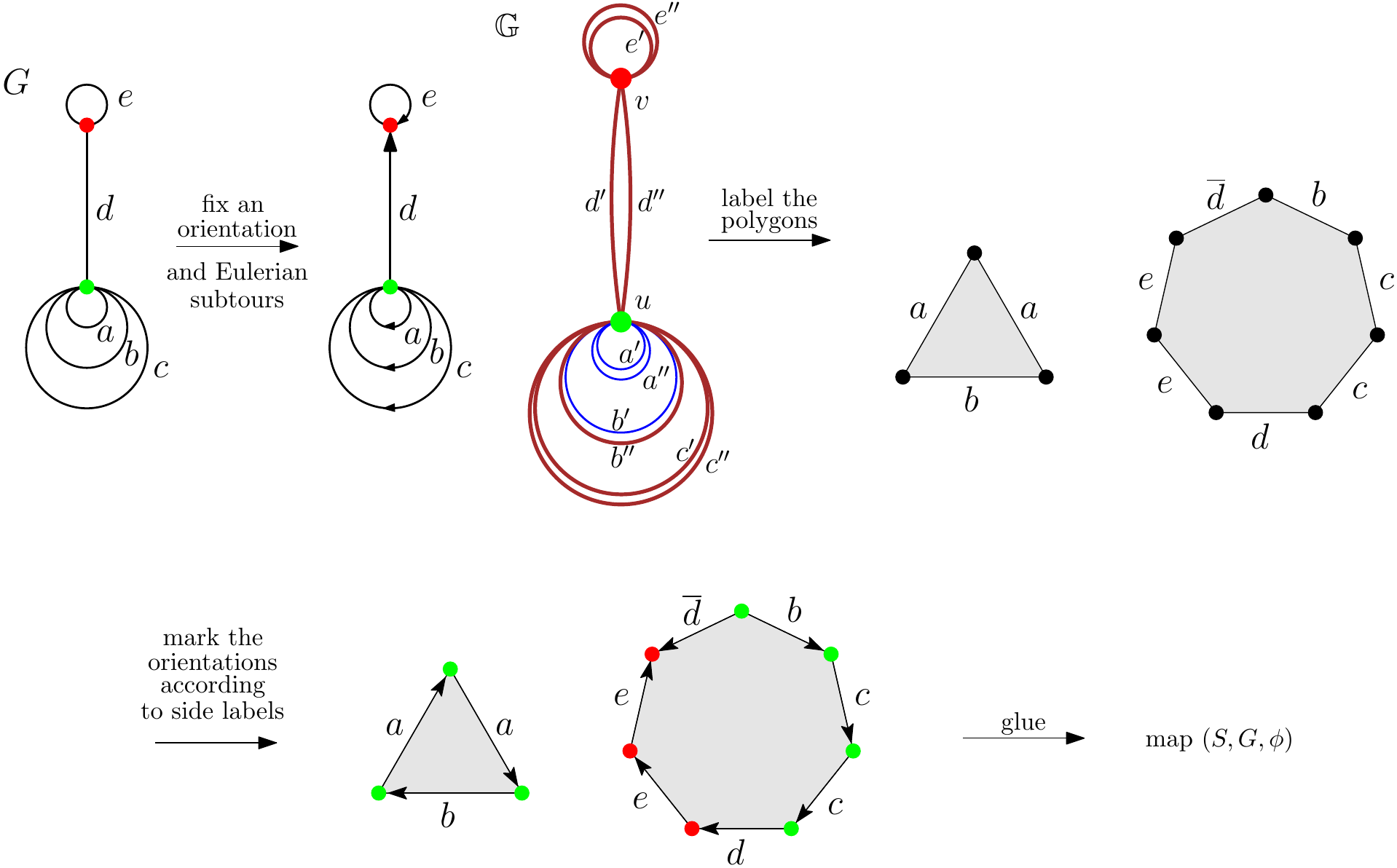}
	\end{center}
	\caption{An example of reversing the construction, using the Eulerian subtours $(u,a',u,a'',u,b'',u)$ and $(u,b'',u,c',u,c'',u,d',v,e',v,e'',v,d'',u)$. There is one $3$-gon and one $7$-gon and the corresponding two words are: $(a\ a\ b)$, $(b\ {c}\  c\ d\ e\ e\ \overline{d})$. After the identification of the sides, we obtain a map $(S,G,\phi)$ with $2$ vertices, $5$ edges, and $2$ countries. The surface $S$ is again of Euler characteristic $-1$, hence $S$ is the non-orientable surface $C_3$.}\label{fig:example3}
\end{figure}

The following transformations preserve both the surface and the map.
\begin{enumerate}[a)]
  \item A cyclic shift of letters in a polygon. For example,
\[
(a\ \overline{a}\ c\ d\ \overline{c}\ b\ e\ f\ \overline{f}\ \overline{e}) \hspace{1cm}\to \hspace{1cm}
(d\ \overline{c}\ b\ e\ f\ \overline{f}\ \overline{e}\ a\
\overline{a}\ c).
 \]
  \item Reorientation of a side:  $a \to \overline{a}$ and $\overline{a}
\to a$ (that is, replace each occurrence of some letter $a$ with $\overline{a}$, and vice versa).

  \item Reorientation of a polygon. That is a combination of two
operations: the reorientation of all the edges of the polygon according
to  b), and replacement of the cyclic order of letters in the polygon
by the inverse one. For example,
$(a\ c\ \overline{c}\ b\ d\ \overline{e})\hspace{1cm}\to \hspace{1cm} (e\ \overline{d}\ \overline{b}\ c\ \overline{c}\ \overline{a})$.
\end{enumerate}

A map and the corresponding surface are called {\it orientable} if there exist orientations of the polygons such that each letter $a$ occurs
once as $a$ and once as $\overline{a}$; see, for example,~\cite{Ringel}.
In other words, it is possible to apply few times operation c) above in such a way that any letter $a$ will occur only in combination $(a,\ \overline{a})$, but not in combination $(a,\ a)$  or $(\overline{a},\ \overline{a})$.

One can find out in~\cite{Ringel} more operations that preserve only
the surface given by a map but not the map itself. These operations
enable us to obtain the classification of surfaces mentioned above.
The surfaces $S_p$ and $C_q$ can be represented by the following {\it normal forms}.
\[
S_p=(a_1\ b_1\ \overline{a}_1\ \overline{b}_1\ a_2\ b_2\
\overline{a}_2\ \overline{b}_2\ \ldots\ a_p\ b_p\ \overline{a}_p\
\overline{b}_p),\ S_0=(a\ \overline{a});\ \
\]
\begin{equation}
C_q=(c_1\ c_1\ c_2\ c_2\ \ldots\ c_q\ c_q);\
p, q \in \{ 1,2,\ldots \}.
\end{equation}
Note that each of these maps contains only one polygon and only one vertex.
For this reason the Euler characteristics are given by formula~\eqref{e:chi-surfaces}.
Surfaces $S_p$ are orientable and $C_q$ are not.

\medskip
\begin{sloppypar}
\noindent{\bf Dual pairs of graphs.}
Consider a map $M=(S,G,\phi)$ with $G = (V,E)$. There is an obvious incidence relation between the countries and edges. The graph of this relation will be called
{\it dual to $G$ on surface $S$} and denoted by $G^*= (F,E^*)$.
Thus two countries are adjacent if and only if they have a common edge (but it is not enough
to have a common vertex). Duality is an involution, in the sense that $(G^*)^*$ is isomorphic to $G$. The dual graph is also connected and it defines a {\it dual map} on the same surface. By definition of the dual graph, there is a bijective correspondence between $E$ and $E^*$, so we will write simply $E^* = E$.
\end{sloppypar}

\medskip
\begin{sloppypar}
\noindent{\bf Geographic plans.} Let $G = (V,E)$ and $H = (F,E)$ be two graphs with a common set of edges. The pair $P=(G,H)$ is called a {\it plan}. Elements of $V$, $E$, and $F$ are called {\it vertices}, {\it edges}, and {\it faces} (or {\it countries}) of the plan. For a plan $P$ as above, we write $V(P) = V$, $E(P) = E$,
and $F(P) = F$. Every map $M$ generates in the obvious way a plan $P=P(M)$ given by
$P =(G,G^*)$ where $G = (V(M),E(M))$ and $G^* = (F(M),E(M))$ are the dual graphs
related to the map. A plan is called {\it geographic} if it is generated by a
map. By duality, a plan $(G,H)$ is geographic if and only if the plan $(H,G)$ is.
\end{sloppypar}

\medskip
\begin{sloppypar}
\noindent{\bf Loops.} In any plan $P=(G,H)$, every edge is incident to one or two
vertices and to one or two faces. An edge $e\in E(P)$ incident to a unique vertex $v\in V(P)$ is called a {\it loop incident to $v$  in $G$}. Similarly,
an edge $e\in E(P)$ incident to a unique face $f\in F(P)$ is called a {\it loop incident to $f$ in $H$}. In case of a geographic plan such an edge $e$ is called a {\it loop at vertex $v$} or an {\it interior edge of country $f$}, respectively.
\end{sloppypar}

\medskip
\begin{sloppypar}
\noindent{\bf Bimatrices of plans.} Let  $P=(G,H)$ be a plan
and let $B_G: E(P) \times V(P) \to \{0,1,2\}$ and $B_H: E(P) \times
F(P) \to \{0,1,2\}$ be the incidence matrices of the graphs $G$ and
$H$, respectively, where an entry of $2$ corresponds to a loop. These matrices have the same set of rows $E(P)$ corresponding to the edges of $P$; the sets of columns $V(P)$ and  $F(P)$ correspond to the vertices and faces of $P$, respectively.
\end{sloppypar}

The sum of the elements in each row is equal to $2$ for both matrices.

The pair $\bfB=(B_G,B_H)$ will be called the {\it bimatrix} of the plan.
Note that a plan $(G,H)$ is uniquely determined by its bimatrix
up to isomorphisms of $G$ and $H$ that keep the bijection between their edge sets.
In other words, the combinatorial structure of a plan given by a bimatrix
$\bfB=(B_G,B_H)$ is not affected by a permutation of the columns of $B_G$,
the columns of $B_H$, or by the same permutation of the
rows of both $B_G$ and $B_H$.

\bigskip
\begin{sloppypar}
\noindent{\bf Examples.}
\begin{enumerate}
  \item Let\label{first-example} $S= S_0$ be a sphere and let $G$ be a graph consisting of two vertices joined by an edge, embedded in $S$. Then $G^*$ has a single vertex and a single edge that is a loop. The bimatrix of the corresponding plan is
       $\bfB = \left( \begin{array}{c|c}11&2\end{array}\right)$.
       Vice versa, let $G$ be the graph consisting of a single vertex and a single edge that is a loop; then  $G^*$ consists of two vertices joined by an edge, and the bimatrix of the corresponding plan is $\bfB = \left( \begin{array}{c|c}2&11\end{array}\right)$.

  \item Let $S$ be a sphere and $G$ be a graph embedded in $S$ consisting of two vertices joined by two parallel edges. Then $G$ is self-dual, that is, $G^*$ is isomorphic to $G$, and the bimatrix of the corresponding plan is $\bfB = \left( \begin{array}{c|c}11&11\\
11&11 \end{array}\right)$.

  \item Bimatrix $\bfB = \left(\begin{array}{c|c}
      11&11
      \end{array}\right)$ is not associated with
any geographic plan $P$, since otherwise the Euler characteristic of the corresponding surface $S$ would be
$\chi(S) = |V(P)|-|E(P)|+|F(P)| = 2 - 1 + 2 = 3 > 2$, which is impossible.

  \item  Let again $S$ be a sphere and let $G$ be a two-edge path embedded in $S$.
Then $G^*$ consists of two loops with a common vertex, and the bimatrix of the corresponding plan is
$\bfB = \left( \begin{array}{c|c}110&2\\
101&2 \end{array}\right)$.

  \item Let again $S$ be a sphere and let $G$ consist of two vertices,
one edge joining them and one loop incident to one of them.
Then $G^*$ is isomorphic to $G$ but the loop in $G^*$ is identified with the (regular) edge in $G$ and vice versa. The bimatrix of the corresponding plan is
$\bfB = \left( \begin{array}{c|c}11&20\\
20&11 \end{array}\right)$.

  \item Bimatrix $\bfB = \left( \begin{array}{c|c}11&11\\
02&20 \end{array}\right)$ is not associated with
any geographic plan. The reason will be explained in Corollary~\ref{cor:even}
in Section~\ref{sec:Edmonds}.

   \item Let $S$ be again a sphere and let $G$ consist of two vertices,
one edge joining them and two loops incident to each. Then $G^*$
consists of two adjacent edges and one loop incident to their common
vertex. The bimatrix of the corresponding plan is
$\bfB = \left( \begin{array}{c|c}
20&110\\
11&020\\
02&011 \end{array}\right)$.

\item Let $S=C_1$ be a projective plane and let $G$ be the graph consisting of one vertex and one loop, embedded in $S$. (Note that up to isomorphism, there is only one way to embed $G$ in $S$ to get a map.) Then $G$ is self-dual, that is, $G^*$ is isomorphic to $G$, and the bimatrix of the corresponding plan is $\bfB = \left( \begin{array}{c|c}
2&2 \end{array}\right)$.
\end{enumerate}
\end{sloppypar}

\medskip
\begin{sloppypar}
\noindent{\bf The vertex--face incidence matrix.}
Given a plan $P = (G,H)$, this matrix is defined as the matrix product $B_{G,H}=B_G^T \cdot B_H$. In other words, $B_{G,H}$ is a matrix the entries of which are the scalar products of two columns corresponding to a vertex $v\in V(P)$ and a face $f\in F(P)$.
\end{sloppypar}

\medskip
\begin{sloppypar}
\noindent{\bf Walks around countries and vertices of a map.}
Let $M = (S,G,\phi)$ be a map and $f$ be one of its countries.
By definition of a map, $f$ is homeomorphic to an open disk. As we go around the boundary of $f$, we obtain a closed walk $W_{f}$ in $G$. This walk contains all edges and vertices incident to country $f$. Given a vertex $v\in V(M)$, let us denote by
$b(v,f)$ the number of occurrences of $v$ in $W_f$. While $b(v,f)$ may be an arbitrary non-negative integer, an edge $e\in E(G)$ may occur at most twice in $W_f$; it occurs exactly twice if and only if $e$ is an interior edge of country $f$.

It is well known that an arbitrary sufficiently small neighborhood
of an arbitrary point of a surface is homeomorphic to an open disk.
Let now $v$ be a vertex of $M$. As we go along the boundary of such a disk,
we obtain a closed walk $W_{v}$ in $G^*$. This walk contains all the edges and faces incident to vertex $v$. Given a face $f\in F(M)$, let us denote by
$b(f,v)$ the number of occurrences of $f$ in $W_v$. While $b(f,v)$ may be an arbitrary non-negative integer, an edge $e\in E(G^*) = E(G)$ may occur at most twice in $W_v$; it occurs exactly twice if and only if $e$ is a loop at vertex $v$.
\end{sloppypar}

\medskip
\begin{sloppypar}
\noindent{\bf Vertex and face graphs of a plan.} Let $P=(G,H)$ be a (not necessarily geographic) plan. Let us fix an arbitrary vertex $v\in V(P)$ and consider the set of all edges and faces incident to it. The subgraph of $H$ formed by them will be denoted by $H_v'$. In this subgraph let us double each edge $e$ corresponding to a loop in $G$.
The obtained graph will be called a {\it vertex graph of the plan} and denoted by $H_v$.
Similarly, for an arbitrary face $f$ of the plan $P$, let $G_f'$ be the subgraph of $G$ formed by all vertices and edges of the plan incident to $f$. The {\it face graph} $G_f$ is the graph obtained from $G_f'$ by duplicating each edge $e$ corresponding to a loop in $H$.
\end{sloppypar}

The following lemma gives a necessary condition for a plan to be geographic.

\begin{lemma}\label{lem:Eulerian-vertex-and-face-graphs}
Let $P = (G,H)$ be a geographic plan. Then, all vertex and
face graphs of $P$ are Eulerian.
\end{lemma}

\begin{proof}
Fix an arbitrary map $M = (S,G,\phi)$ generating $P$. Consider a face $f\in F(M)$. We will show that $W_f$ is an Eulerian trail in the face graph $G_f$. By construction, $W_f$ is a closed walk (but not necessarily a trail) in $G_f'$. Furthermore, every edge $e$ of $G_f'$ occurs in $W_f$ either once or twice; it occurs exactly twice if and only if $e$ is an interior edge of country~$f$. Recall also that the graph $G_f$ is obtained from $G_f'$ by duplicating each edge $e$ that is a loop in $H = G^*$, that is, by duplicating each interior edge of $f$. It follows that every edge of $G_f$ occurs in $W_f$ exactly once, that is, $W_f$ is an Eulerian trail in $G_f$. Similarly, it can be shown that for every vertex $v\in V(M)$, the walk $W_v$ is an Eulerian trail in $H_v$.
\end{proof}

\begin{sloppypar}
\section{Edmonds' characterization and its consequences}\label{sec:Edmonds}
\end{sloppypar}

\begin{sloppypar}
\noindent{\bf Properties of plans.}
A plan $P=(G,H)$ will
be called:
\begin{enumerate}[a)]
  \item\label{condition:b} {\it connected} if $G$ and $H$ are connected;
\item\label{condition:d} {\it locally Eulerian} if all its vertex and face graphs are Eulerian;
\item\label{condition:a} {\it even} if all the elements of its
vertex--face incidence matrix $B_{G,H}$ are even.
\end{enumerate}
\end{sloppypar}

\bigskip
The following characterization of geographic plans was given \hbox{by~\citet{Edmonds}}.

\begin{theorem}\label{thm:Edmonds}
A plan is geographic if and only if it is connected and locally Eulerian.
\end{theorem}

The ``only if'' part follows from the definition of a map and Lemma~\ref{lem:Eulerian-vertex-and-face-graphs}. The ``if'' part of the theorem and its proof by~\citet{Edmonds} can be seen as an extension of Euler's ``K\"onigsberg Bridge'' theorem and its proof from~\cite{euler1736solutio}.

\medskip
A connected and locally Eulerian plan can be generated by different maps on the same surface or even on different surfaces; see the examples below. Nevertheless, the Euler characteristic of all surfaces of the maps generating the plan is unambiguously determined by the plan, according to equality $\chi=|V(P)|-|E(P)|+|F(P)| = n-\ell+m$, see~\eqref{e:chi}. Thus in any case there exist at most two such surfaces, there is only one if  $\chi = 2$  or $\chi =2i+1 < 2$, and there is none if $\chi > 2$. In particular, inequality $|V(P)|-|E(P)|+|F(P)| \le 2$  holds for any plan $P$ which is connected and locally Eulerian, for otherwise there exists no surface for a corresponding map.

\medskip
The following results, including a simple necessary condition for a plan to be geographic, were announced in~\cite{Gurvich-RRR,GurvichShabat}.

\begin{theorem}\label{thm:even}
Let $P = (G,H)$ be a plan. Then, for an arbitrary vertex \hbox{$v\in V(P)$} and face $f\in F(P)$, we have
\begin{equation}\label{e:degrees-in-walks-plans}
B_{G,H}(v,f) = d_{H_v}(f)= d_{G_f}(v) = \sum_{e\in E(P)} B_{G}(e,v) \cdot B_{H}(e,f)
\,.
\end{equation}
Each edge $e\in E(P)$ occurs exactly twice in the sets
$\{E(H_{v})\mid v\in V(P)\}$ as well as in the sets
$\{E(G_{f})\mid f\in F(P)\}$.
Furthermore, if $P$ is a geographic plan generated by a map $M = (S,G,\phi)$, with $H = G^*$, then, for an arbitrary vertex $v\in V(M)$ and country $f\in F(M)$ we have
\begin{equation}\label{e:degrees-in-walks}
2b(v,f)=2b(f,v)=B_{G,G^*}(v,f) = \sum_{e\in E(M)} B_{G}(e,v) \cdot B_{G^*}(e,f)\,.
\end{equation}
\end{theorem}

\begin{proof}
Fix $P$, $v$, and $f$ as above. The equality $B_{G,H}(v,f) = \sum_{e\in E(P)} B_{G}(e,v) \cdot B_{H}(e,f)$ follows from the definition of the bimatrix of the plan, $B_{G,H}$.

Let us justify the equality $d_{G_f}(v) = \sum_{e\in E(P)} B_{G}(e,v) \cdot B_{H}(e,f)$.
For all $i,j\in \{0,1,2\}$, let $E_{i,j}$ be the set of edges $e\in E(P)$ such that
$B_G(e,v) = i$ and $B_H(e,f) = j$.
Then the sets $\{E_{i,j}\mid i,j\in \{0,1,2\}\}$ partition $E(P)$.
We have

\begin{eqnarray}\label{e:sums}
\begin{aligned}
\sum_{e\in E(P)} B_{G}(e,v) \cdot B_{H}(e,f)
 &={} \sum_{(i,j)\in \{0,1,2\}^2}\Big(\sum_{e\in E_{i,j}} B_{G}(e,v) \cdot B_{H}(e,f)\Big)\\
 &={}
 |E_{1,1}| + 2|E_{1,2}|+ 2|E_{2,1}|+ 4|E_{2,2}|\\
&={}
 \left(|E_{1,1}| + 2|E_{2,1}|\right) +
 2\left(|E_{1,2}|+ 2|E_{2,2}|\right)\,.
\end{aligned}
\end{eqnarray}
Recall that the degree of a vertex in a graph is defined as the number of edges incident to it, with loops counted twice. Since $G_f'$ is the subgraph of $G$ formed by all the vertices and edges incident to $f$, we thus have $$d_{G_f'}(v) = (|E_{1,1}| + |E_{1,2}|) +
2(|E_{2,1}| + |E_{2,2}|) = (|E_{1,1}| + 2|E_{2,1}|) +
(|E_{1,2}| + 2|E_{2,2}|)\,.$$
Furthermore, since $G_f$ is the graph obtained from $G_f'$ by duplicating each edge $e$ corresponding to a loop in $H$, we have
\begin{equation}\label{e:sums-2}
d_{G_f}(v) = \left(|E_{1,1}| + 2|E_{2,1}|\right) +
 2\left(|E_{1,2}|+ 2|E_{2,2}|\right)\,.
\end{equation}
Combining~\eqref{e:sums} and~\eqref{e:sums-2} we obtain
$d_{G_f}(v) = \sum_{e\in E(P)} B_{G}(e,v) \cdot B_{H}(e,f)$, as claimed.

Equality
$d_{H_v}(f) = \sum_{e\in E(P)} B_{G}(e,v) \cdot B_{H}(e,f)$
follows by duality. Thus, equations~\eqref{e:degrees-in-walks-plans} are proved.

\medskip
Next we show that each edge $e\in E(P)$ occurs exactly twice in the sets
$\{E(H_{v})\mid v\in V(P)\}$. If $e$ is a loop in $G$, say at vertex $v$, then $e$ will appear duplicated in the graph $H_v$ and will not appear in any graph $H_w$ for $w\in V(P)\setminus \{v\}$. If $e$ is a not loop in $G$, say $e$ connects distinct vertices $v$ and $w$, then $e$ will appear exactly once in each of the graphs $H_v$ and $H_w$ and will not appear in any graph $H_u$ for all $u\in V(P)\setminus \{v,w\}$. Thus, in either case, $e$ occurs exactly twice in the sets $\{E(H_{v})\mid v\in V(P)\}$. The fact that each edge $e\in E(P)$ occurs exactly twice in the sets $\{E(G_{f})\mid f\in F(P)\}$ follows by duality.

\medskip
It remains to prove equations~\eqref{e:degrees-in-walks}.
Suppose that $P = (G,G^*)$ is a geographic plan generated by a map $M$
and let $v\in V(M)$ and $f\in F(M)$. Let us justify the equality
$2b(v,f) = \sum_{e\in E(M)} B_{G}(e,v) \cdot B_{G^*}(e,f)\,.$
Fix a cyclic order of occurrences of $v$ on $W_f$, say $v^1,\ldots, v^{b}$, where $b = b(v,f)$. For each $i\in \{1,\ldots, b\}$, let us denote by
$e^{i,-}$ the edge of $G$ immediately preceding $v^i$ on $W_f$ and by
$e^{i,+}$ the edge of $G$ immediately following $v^i$ on $W_f$.
The sequence of edges
$$\sigma^v = (e^{1,-},e^{1,+}, \ldots, e^{b,-},e^{b,+})$$
contains precisely the edges incident to $v$ and $f$.

Let $X$ be the number of triples $(i,j,e)$ such that $i\in \{1,\ldots,b\}$,
$j\in \{-,+\}$ and $e = e^{i,j}$. We now count $X$ in two ways. Every pair $i\in \{1,\ldots,b\}$ and
$j\in \{-,+\}$ defines a unique triple $(i,j,e)$ in $X$, by setting $e = e^{i,j}$.
Thus, $|X| = 2b$.

The definition of incidence matrix $B_{G^*}$ implies that for each edge $e\in E(M)$, the number of occurrences of $e$ on $W_f$ equals $B_{G^*}(e,f)$. Let $E_v$ denote the set of edges incident to $v$. Consider the sequence of edges forming $W_f$ and its subsequence $\sigma$ formed by the edges in $E_v$. This subsequence contains each edge $e\in E_v$ precisely $B_{G^*}(e,f)$ times.
If $e$ is a loop at $v$, then each of the $B_{G^*}(e,f)$ occurrences of $e$ in this subsequence corresponds to exactly two elements of $X$, namely to two triples of the form $(i,+,e)$ and $(i+1,-,e)$ for some $i\in \{1,\ldots, b\}$ (where $b+1 = 1$).
If $e$ is a regular edge, then each of the $B_{G^*}(e,f)$ occurrences of $e$ in this subsequence corresponds to exactly one element of $X$, namely to a triple of the form $(i,+,e)$ or $(i,-,e)$ for some $i\in \{1,\ldots, b\}$. For an edge $e\in E_v$, we have $B_{G}(e,v) = 2$ if $e$ is a loop at $v$ and $B_{G}(e,v) = 1$, otherwise.
Thus, for each edge $e\in E_v$, all the occurrences of $e$ in $\sigma$ generate exactly $B_{G}(e,v)\cdot B_{G^*}(e,f)$ triples in $X$. Consequently,
$$|X| = \sum_{e\in E_v}B_{G}(e,v)\cdot B_{G^*}(e,f) = \sum_{e\in E(M)}B_{G}(e,v)\cdot B_{G^*}(e,f)$$ and equality
$2b(v,f) = \sum_{e\in E(M)} B_{G}(e,v) \cdot B_{G^*}(e,f)$
is proved.

Equality
$2b(f,v) = \sum_{e\in E(M)} B_{G}(e,v) \cdot B_{G^*}(e,f)$
follows by duality. This shows equations~\eqref{e:degrees-in-walks} and completes the proof of the theorem.
\end{proof}

\begin{corollary}\label{cor:even}
Every geographic plan is even.
\end{corollary}

\medskip
For example, the bimatrix
$\bfB =
\left(\begin{array}{c|c}
      11&11
      \end{array}\right)$ is not associated with
any geographic plan, and neither is
$\bfB = \left( \begin{array}{c|c}11&11\\
02&20 \end{array}\right)$.

\bigskip
\begin{sloppypar}
\noindent{\bf Building a map generating a given geographic plan.}
Let $P = (G,H)$ be a geographic plan. By Theorem~\ref{thm:Edmonds},
$P$ is connected and locally Eulerian. Label the vertices of $G$ as $v_1,\ldots, v_n$ and consider for each face $f\in F(P)$
the following procedure, which will associate to $f$ a point- and side-labeled polygon $Q_f$ in the plane, with oriented sides.
Fix an Eulerian trail $C_f$ in $G_f$. By definition, $C_f$ is an alternating sequence of vertices and edges incident with face $f$, of length $\ell_f := |E(G_f)|$. Fix an $\ell_f$-gon $Q_f$ in the plane. Label the $\ell_f$ sides of $Q_f$ and the $\ell_f$ points on its boundary with edges and vertices of $G_f$ so that when the polygon is traversed clockwise, the labels spell out trail $C_f$. Each side $s$ of $Q_f$ is now labeled with an edge $e$ of $G$ with endpoints $v_i$ and $v_j$ where $i\le j$. We orient $s$ from $v_i$ to $v_j$. In particular, we orient $s$ arbitrarily if $e$ is a loop. This completes the description of the procedure for each face $f\in F(P)$. We further assume that the polygons $\{Q_f\mid f\in F(P)\}$ are pairwise disjoint.

By Theorem~\ref{thm:even}, each edge $e\in E(P)$ occurs exactly twice in the sets
$\{E(G_{f})\mid f\in F(P)\}$. Therefore, each edge $e\in E(P)$ labels exactly two sides of polygons in the set $\{Q_f\mid f\in F(P)\}$. This naturally partitions the set of all the polygon sides into pairs. In particular, the sum of the numbers $\ell_f$ over all $f\in F(P)$ is even. Furthermore, for any proper subset of the set of all polygons $\{Q_f\mid f\in F(P)\}$ there exists a side label that only occurs once in the polygons of this subset, since otherwise $H$ would be disconnected.

We now glue the polygons along the pairs of sides labeled by the same edge of $G$, in accordance with their orientations as defined above. We thus obtain a surface $S$, but not necessarily a map of the form $(S,G,\phi)$. There may be two reasons for this.
First, after gluing the sides, some vertices of $G$ may still label more than one point of $S$. Second, it may happen that distinct vertices of $G$ become identified after gluing. However, as shown by Edmonds in his proof of Theorem~\ref{thm:Edmonds}, these problematic situations can always be avoided by an appropriate choice of the Eulerian trails $C_f$ in the face graphs $G_{f}$ and/or appropriate orientations of sides of the polygons.
\end{sloppypar}

\bigskip
\begin{sloppypar}
\noindent{\bf More examples.} We illustrate the above construction with several examples.
\begin{enumerate}
  \item[9.] Consider the plan $P = (G,H)$ given by the bimatrix
  $\bfB =
      \left(\begin{array}{c|c}
      2&2
      \end{array}\right)$.
Then $|V(P)| = |E(P)| = |F(P)| = \chi =1$.
The face graph $G_f$ of the unique face $f$ consists of the unique
vertex $v$ and two copies of the unique edge $e$, which is a loop.
Denoting the two copies of $e$ in $G_f$ by $e'$ and $e''$, there is an essentially unique Eulerian trail $C_f$ in the face graph $G_{f}$, namely $(v,e',v,e'',v)$.
In the $2$-gon $Q_f$ we can orient the two sides in essentially two ways:
either a) in the same direction (both clockwise or both counterclockwise) or b) oppositely (one clockwise and one counterclockwise). In case a), we obtain a map on the projective plane $C_1$. It is easy to check that this map has only one vertex, and thus really generates plan $P$.  In case b), we obtain a sphere $S_0$ with two vertices, which, however, must be identified, because there is only one vertex in the plan. In this case, we do not obtain a map.

Note that in terms of word representations of surfaces, the map given in case a) above can be described as $(a\ a)$, while the sequence $(a\ \overline{a})$ corresponding to case b) is not valid.

\item[10.] Let us return to Example 1 on p.~\ref{first-example}. The bimatrix of the plan is $\bfB = \left( \begin{array}{c|c} 11&2\end{array}\right)$.
Let $V(P) = \{v_1,v_2\}$, $E(P) = \{e\}$, and $F(P) = \{f\}$. The graph $G_f$ has two vertices, $v_1$ and $v_2$, joined by two copies of $e$, say $e'$ and $e''$. There is an essentially unique Eulerian trail $C_f$ in the face graph $G_{f}$, namely $(v_1,e',v_2,e'',v_1)$.
    Again, in the $2$-gon $Q_f$ we can orient the two sides in essentially two ways: either a) in the same direction (both clockwise or both counterclockwise) or b) oppositely (one clockwise and one counterclockwise). In case a), we obtain a projective plane $C_1$. However, the resulting map has only one vertex, and thus does not correspond to the original plan, which has two vertices. In case b), we obtain a map on a sphere $S_0$ with two vertices, which does generate plan $P$.

    In terms of word representation of surfaces, the sequence given in case a) above can be described as $(a\ a)$, which is not valid, while the valid map given in case b) can be described with the sequence $(a\ \overline{a})$.
\end{enumerate}

Given a geographic plan $P = (G,H)$, we say that the procedure as above is {\it valid} if it indeed results in a map $M = (S,G,\phi)$ generating $P$. This happens if and only if each vertex of $G$ corresponds to a unique point on surface $S$.
In this case, the corresponding word representation of the map will also be called {\it valid} for $P$.
As the above examples show, not all word representations of the resulting surface $S$ are valid for the given plan. For the sake of simplicity, in what follows we will denote the edges of the graph $G$ and sides of the polygons with the corresponding letters in the resulting word representations of the surface.

\medskip
Let us call a plan $P = (G,H)$ {\it simple} if all its vertex graphs $H_v$ and face graphs $G_f$ are cycles. It is easy to see that a geographic simple plan corresponds to a unique map, and hence a unique surface. An example of a simple geographic plan is given by bimatrix
$$\bfB =
\left(\begin{array}{cccc|cccc}
1&1&0&0&1&1&0&0\\
1&1&0&0&0&0&1&1\\
1&0&1&0&1&0&0&1\\
1&0&1&0&0&1&1&0\\
0&1&0&1&0&1&1&0\\
0&1&0&1&1&0&0&1\\
0&0&1&1&1&1&0&0\\
0&0&1&1&0&0&1&1
\end{array}\right)\,.$$
This plan is generated by a unique map on the torus.
Each of the four vertex graphs and four face graphs is a $4$-cycle. We leave it to the reader to construct this map as well as a similar example of a simple plan generated from a unique map on the Klein bottle.

\medskip
On the other hand, some geographic plans correspond to multiple maps, on different surfaces. The next example illustrates this for $\chi = 0$.

\medskip
\begin{enumerate}
  \item[11.] Let $P = (G,H)$ be a plan given by the bimatrix
$\bfB = \left( \begin{array}{c|c}
2&2\\
2&2\end{array}\right)$.
Then $|V(P)| = |F(P)| = 1$ and $|E(P)| = 2$, hence $\chi = 0$. The face graph $G_{f}$ contains two doubled loops $a$ and $b$.

Up to map-preserving transformations there exist eight different word representations of the resulting surface:
$(a\ b\ \overline{a}\ \overline{b})$,
$(a\ b\ \overline{b}\ \overline{a})$, $(a\ a\ b\ \overline{b})$, $(a\ \overline{a}\ b\ b)$,
$(a\ b\ a\ b)$, $(a\ a\ b\ b)$, $(a\ b\ \overline{a}\ b)$, and $(a\ b\ a\ \overline{b})$.
The first one generates the normal form map on the torus $S_1$.
The second one generates a sphere $S_0$ with three vertices, which, however, must be
identified because there is only one vertex in the plan. In this case, we do not obtain a map. The third, the  fourth, and the fifth representations all generate the projective plane $C_1$ with two identical vertices. Thus, they do not generate maps. Finally, the last three representations generate maps on the Klein bottle $C_2$ (the first of them in the normal form).
\end{enumerate}
\end{sloppypar}

Further examples of geographic plans and the corresponding maps can be found in~\cite{Gurvich-RRR}.

\section{Degree sequences of geographic plans}\label{sec:degree-sequences}

Given a graph $G$, the graph obtained from $G$ by duplicating each edge will be called the {\it double graph of $G$} and denoted by $\GG$. Given a connected graph $G=(V,E)$, $|V|=n$, we call the sequence $(d_G(v)\mid v\in V)$ its {\it degree sequence}. If $\bfd=(d_1,\ldots,d_n)=(d_G(v)\mid v\in V)$ (after an appropriate of permutation of the vertices) then we call $G$ a \emph{realization} of $\bfd$. The {\it degree sequence} of a plan $P = (G,H)$ is the pair $(\bfd;\bft)$ where $\bfd$ is the degree sequence of $G$ and $\bft$ is the degree sequence of $H$.

Let $G = (V,E)$ be a connected graph and $\bft=(t_1,\ldots,t_m)\in\ZZ_+^m$ a sequence of non-negative integers such that $\sum_{j = 1}^mt_j = 2|E|$.
Let us consider a partition $\mathcal{E} = \{E_1,\ldots, E_m\}$ of the edge set of the double graph $\GG$. The indices $1,\ldots, m$ will be called {\it colors} and the sets $E_1,\ldots, E_m$ the {\it color classes}.
To each vertex $v\in V(G)$ we associate a graph $H_v^{\mathcal{E}}$ on vertex set $$V(H_v^{\mathcal{E}}) = \{j\in \{1,\ldots, m\}\mid \textrm{there exists an edge $e\in E_j$ incident with }v\}$$
such that the edge set of $H_v^{\mathcal{E}}$ consists of those color pairs $\{j,k\}$ with $j,k\in V(H_v^{\mathcal{E}})$ and $j\neq k$ for which there exists an edge $e\in E(G)$ incident with $v$ such that $e'\in E_j$ and $e''\in E_k$ where $e'$ and $e''$ are the two copies of $e$ in $\GG$.

The partition $\mathcal{E}$ is said to be:
\begin{itemize}
  \item {\it Eulerian} if each of the graphs $(V,E_j)$ defined by the color classes is Eulerian, $j= 1,\ldots, m$;
 \item a {\it $\bft$-partition} if $|E_j| = t_j$ for all $j\in \{1,\ldots, m\}$;
 \item {\it locally connected} if for each vertex
 $v\in V(G)$ the graph $H_v^{\mathcal{E}}$ is connected.
\end{itemize}

\begin{theorem}\label{thm:geographic-degree-sequences}
Consider two integer sequences $\bfd=(d_1,\ldots,d_n)\in\ZZ_+^n$ and $\bft=(t_1,\ldots,t_m)\in\ZZ_+^m$.
Then, the following statements are equivalent.
\begin{enumerate}[a)]
  \item $(\bfd;\bft)$ is the degree sequence of a geographic plan.
  \item $(\bft;\bfd)$ is the degree sequence of a geographic plan.
\item There exists a connected realization $G$ of $\bfd$ such that the double graph $\GG$ has a locally connected Eulerian $\bft$-partition.
\item There exists a connected realization $H$ of $\bft$ such that the double graph $\HH$ has a locally connected Eulerian $\bfd$-partition.
\end{enumerate}
\end{theorem}

\begin{proof}
The equivalences between statements $a)$ and $b)$
and between statements $c)$ and $d)$ follow from the fact that a plan $(G,H)$ is geographic if and only if the plan $(H,G)$ is.

\medskip
Next we prove the implication $a)\Rightarrow c)$. Suppose that $(\bfd;\bft)$ is the degree sequence of a geographic plan $P = (G,H)$. Then $G$ is a connected realization of $\bfd$.
It remains to show that the double graph $\GG$ has a locally connected Eulerian $\bft$-partition. Let us identify the vertex set of $H$ (that is, the face set of the plan) with
the set $\{1,\ldots, m\}$ and consider its elements as colors.
Using these colors, we now partition the edges of
$\GG$, as follows. For each edge $e\in E(P)$ and each color
$f\in F(P) = \{1,\ldots, m\}$, we assign color $f$ to exactly $B_H(e,f)$ copies of $e$ in $\GG$. (Recall that $B_H(e,f)\in \{0,1,2\}$.) We do this in such a way that every edge of $\GG$ gets assigned a unique color. This is possible since for each edge $e\in E(P)$, the sum of the entries of the corresponding row of $B_H$ is equal to $2$. (Equivalently, each edge $e\in E(P)$ occurs exactly twice in the sets $\{E(G_{f})\mid f\in F(P)\}$, see Theorem~\ref{thm:even}.)

Notice that for each color $f\in F(P)$, the corresponding color class is the same as the edge set of the face graph $G_f$. Thus, by Lemma~\ref{lem:Eulerian-vertex-and-face-graphs}, each color class defines an Eulerian subgraph of $\GG$. Thus, the above coloring procedure defines an Eulerian partition $\mathcal{E} = \{E_1,\ldots, E_m\}$ of the double graph $\GG$. For each color $f\in \{1,\ldots, m\}$, we also have $|E_f| = \sum_{e\in E(P)}B_H(e,f) = t_f$, where the last equality holds since $H$ is a realization of $\bft$. Thus, $\mathcal{E}$ is a $\bft$-partition.

Furthermore, Lemma~\ref{lem:Eulerian-vertex-and-face-graphs}
implies that all vertex graphs $H_v$, $v\in V(G)$, of $P$ are Eulerian, and thus connected. Since for each $v\in V(G)$, the graphs $H^{\mathcal{E}}_v$ and $H_v$ have the same vertex set
and two vertices of $H^{\mathcal{E}}_v$ are adjacent if and only if they are adjacent in $H_v$, the graph $H^{\mathcal{E}}_v$ is connected as well. This shows that the partition
$\mathcal{E}$ is also locally connected and establishes the implication $a)\Rightarrow c)$.

\medskip
Finally, we prove the implication $c)\Rightarrow a)$.
Suppose that there exists a connected realization $G$ of $\bfd$ such that the double graph $\GG$ has a locally connected Eulerian $\bft$-partition $\mathcal{E}$. We complete the proof by constructing a graph $H$ such that $P = (G,H)$ is a geographic plan with degree sequence $(\bfd;\bft)$.

Since $\mathcal{E}$ is a $\bft$-partition, we have $\mathcal{E} = \{E_1,\ldots, E_m\}$ such that $|E_j| = t_j$ for all $j\in \{1,\ldots, m\}$. For each $j\in \{1,\ldots, m\}$, let $V_j$
be the set of vertices of $\GG$ incident with an edge in $E_j$.

The graph $H$ is defined as follows. The vertex set of $H$ is $\{1,\ldots, m\}$. The edge set of $H$ is in bijective correspondence with the edge set of $G$. For each edge $e$ of $G$ there are two copies $e'$ and $e''$ of $e$ in $\GG$; we say that $e$ connects colors $f,g\in \{1,\ldots, m\}$ (possibly with $f = g$) in $H$ if $e'\in E_f$ and $e''\in E_g$ or vice versa. Clearly, $P = (G,H)$ is a plan.

Consider an arbitrary color $f\in \{1,\ldots, m\}$.
By the above construction, the degree of $f$ in $H$ is exactly $|E_f| = t_f$. This shows that $H$ is a realization of $\bft$ and hence $(\bfd;\bft)$ is the degree sequence of $P$.

It remains to prove that $P$ is geographic. By Theorem~\ref{thm:Edmonds}, this holds if and only if $P$ is connected and locally Eulerian.

Graph $G$ is connected by assumption. To see that $H$ is connected, let us consider two colors $f,g\in V(H)$.
Since $G$ is connected, so is $\GG$, and hence there exists a path
$(v_1,e_1,v_2,\ldots, e_{k-1},v_k)$ in $\GG$ such that
$v_1\in V_f$ and $v_k\in V_g$. Furthermore, let $e_0$ be an edge in $E_f$ incident with $v_1$, let $e_k$ be an edge in $E_g$ incident with $v_k$, and for each $j\in \{0,1,\ldots, k\}$ let $h_j$ be a color such that $e_j\in E_{h_j}$ where $h_0 = f$ and $h_k = g$. Since $\mathcal{E}$ is locally connected, for each $j\in \{1,\ldots,k\}$, colors $h_{j-1}$ and $h_j$ are connected by a path in the graph $H_{v_j}^{\mathcal{E}}$ and, since
$H_{v_j}^{\mathcal{E}}$ is a subgraph of $H$, they are connected in $H$, too. Consequently, $f = h_0$ and $g = h_k$ are also connected in $H$. This shows that $H$ is connected and hence the plan $P = (G,H)$ is connected.

It remains to prove that $P$ is locally Eulerian. In other words, its vertex and face graphs are Eulerian. The face graphs $G_f$ of $P$ are Eulerian since
$G_f$ is isomorphic to $(V_f,E_f)$ and the partition $\mathcal{E}$ is Eulerian.
Consider now a vertex graph $H_v$ for some $v\in V(G)$.
Since the partition ${\mathcal{E}}$ is locally connected, the graph $H^{\mathcal{E}}_v$ is connected. This implies that the vertex graph $H_v$ is connected, since the adjacency relation of these two graphs is the same. By Theorem~\ref{thm:Eulerian}, it is enough to show that all vertices of $H_v$ have even degrees. Fix a vertex $f\in V(H_v)$.
Then $v\in V_f$. Since the graph $(V_f,E_f)$ is Eulerian, the degree of
$v$ in $(V_f,E_f)$ is even. Due to the bijection between $E(G)$ and $E(H)$,
this degree is the same as the degree of $f$ in $H_v$.
This completes the proof.
\end{proof}

\section{Examples of realizable sequence pairs}

Given two integer sequences $\bfd=(d_1,\ldots,d_n)\in\ZZ_+^n$ and
$\bft=(t_1,\ldots,t_m)\in\ZZ_+^m$, we say that the pair $(\bfd;\bft)$ is
\emph{realizable} if it is the degree sequence of a geographic plan.
We now apply Theorem~\ref{thm:geographic-degree-sequences} to
construct infinite families of realizable sequence pairs.

\begin{proposition}\label{prop:nn532n-4}
Let $n\geq  4$ be an integer and consider the sequences $\bfd=(n,n)\in \ZZ_+^{2}$ and
$\bft=(5,3,2^{n-4})\in \ZZ_+^{n-2}$. Then the pair $(\bfd;\bft)$ is realizable.
(Note that $\chi(\bfd;\bft)=0$.)
\end{proposition}

\begin{proof}
Fix an integer $n\geq 4$ and let $H$ be the graph obtained as follows. Take two vertices $u$ and $v$, put a loop on $u$,
put three edges $a,b,c$ between $u$ and $v$, and subdivide edge $a$ by $n-4$ new vertices.
Clearly, $H$ is a connected realization of $(7,5,2^{2n-6})$.
Consider the sets $E_1,E_2$ of the edges of the double graph $\HH$ defined as follows:
$E_1$ contains one copy of the path of length $n-3$ corresponding to edge $a$, both copies of edge $b$, and one copy of edge $c$;
$E_2$ contains the other copy of the path corresponding to edge $a$, the other copy of edge $c$, and both copies of the loop at $u$.
Then, $\{E_1,E_2\}$ is a locally connected Eulerian $(n,n)$-partition of $\HH$. By Theorem~\ref{thm:geographic-degree-sequences},
the pair $(\bfd;\bft)$ is realizable.
\end{proof}

\begin{sloppypar}
\begin{proposition}\label{prop:nnnn752n-6}
Let $n\geq  3$ be an integer and consider the sequences $\bfd=(n,n,n,n)\in \ZZ_+^{4}$ and
$\bft=(7,5,2^{n-6})\in \ZZ_+^{n-4}$. Then the pair $(\bfd;\bft)$ is realizable.
(Note that $\chi(\bfd;\bft)=0$.)
\end{proposition}
\end{sloppypar}

\begin{proof}
We first consider the case $n = 3$. Let $H$ be the graph consisting of two vertices $u$ and $v$, having two loops on $u$, one loop on $v$, and three edges $a,b,c$ between $u$ and $v$. This is a connected realization of $(7,5)$. Consider the sets $E_1,E_2,E_3,E_4$ of the edges of the double graph $\HH$ defined as follows: $E_1$ contains both copies of one of the loops at $u$ and one copy of the other one, $E_2$ contains the remaining copy of a loop at $u$ along with $a$ and $b$, $E_3$ contains $b$, $c$, and one copy of the loop at $v$, and $E_4$ contains $a$, $c$, and the other copy of the loop at $v$. Then, $\{E_1,E_2,E_3,E_4\}$ is a locally connected Eulerian $(3,3,3,3)$-partition of $\HH$.
By Theorem~\ref{thm:geographic-degree-sequences}, the pair $(3,3,3,3;7,5)$ is realizable.
	
Now let $n\geq 4$. Let $H$ be the graph obtained as follows. Take two vertices $u$ and $v$, put a loop on $u$, and put five edges $a,b,c,d,e$ between $u$ and $v$. We then subdivide edge $a$ by $n-2$ new vertices and edge $b$ by $n-4$ new vertices. Clearly, $H$ is a connected realization of $(7,5,2^{2n-6})$. Consider the sets $E_1,E_2,E_3,E_4$ of the edges of the double graph $\HH$ defined as follows:
$E_1$ contains one copy of the path of length $n-1$ corresponding to edge $a$ and one copy of edge $c$;
$E_2$ contains the other copy of the path corresponding to edge $a$ and one copy of edge $d$;
$E_3$ contains one copy of the path of length $n-3$ corresponding to edge $b$ and one copy of each edge $c$, $d$, $e$;
$E_4$ contains the other copy of the path corresponding to edge $b$, one copy of edge $e$, and both copies of the loop at $u$.
Then, $\{E_1,E_2,E_3,E_4\}$ is a locally connected Eulerian $(n,n,n,n)$-partition of $\HH$.
Thus, the pair $(\bfd;\bft)$ is realizable by Theorem~\ref{thm:geographic-degree-sequences}.
\end{proof}

\begin{sloppypar}
\begin{proposition}\label{prop:n2kak}
Let $n$, $k$, $a$ be positive integers such that \hbox{$n\geq a\geq 4$} and $a$ is even,
and consider the sequences $\bfd=(n^{2k})\in \ZZ_+^{2k}$ and
\hbox{$\bft=(ak+1,ak-1, 2^{k(n-a)})\in \ZZ_+^{k(n-a)+2}$}.
Then the pair $(\bfd;\bft)$ is realizable.
(Note that $\chi(\bfd;\bft)=2-(a-2)k$ is even and non-positive.)
\end{proposition}
\end{sloppypar}

\begin{proof}
Let $H$ be the graph obtained as follows. Take two vertices $u$ and $v$, put a loop $e_1$ on $u$, and put $ak-1$ edges $e_{2},\ldots, e_{ak}$ connecting $u$ and $v$. Let us now subdivide $k$ of these connecting edges, say $e_{(a-1)k+1},\ldots, e_{ak}$, by introducing $n-a$ new vertices on each, and let $P_1,\ldots, P_k$ be the resulting $u,v$-paths. Clearly, $H$ is a connected realization of the degree sequence $\bft$.
To complete the proof, we consider two cases depending on the value of $k$.

Suppose first that $k = 1$. In this case, consider the sets $E_1,E_2$ of the edges of the double graph $\HH$ defined as follows:
$E_1$ contains one copy of path $P_1$, two copies of loop $e_1$, and one copy of each edge $e_2,\ldots, e_{a-2}$;
$E_2$ contains one copy of path $P_1$, one copy of each edge $e_2,\ldots, e_{a-2}$, and two copies of edge $e_{a-1}$.
Note that the degrees of $u$ and $v$ in $(V(H),E_1)$ equal $a+2$ and $a-2$, respectively, while both vertices have degree
$a$ in $(V(H),E_2)$. It follows that $\{E_1,E_2\}$ is a locally connected Eulerian $(n,n)$-partition of $\HH$.

Suppose now that $k>1$. In this case, consider the sets $E_1,\ldots, E_{2k}$ of the edges of the double graph $\HH$ defined as follows:
$E_1$ contains one copy of path $P_1$, both copies of loop $e_1$ and one copy of each of the edges $e_3,\ldots, e_{a-1}$;
$E_2$ contains one copy of path $P_2$, both copies of edge $e_2$ and one copy of each of the edges $e_3,\ldots, e_{a-1}$;
for all $i\in \{2,\ldots, k\}$, set $E_{2i-1}$ contains one copy of path $P_i$ and
one copy of each of the edges $e_{(i-1)(a-1)+1},\ldots, e_{i(a-1)}$;
for all $i\in \{2,\ldots, k-1\}$, set $E_{2i}$ contains one copy of path $P_{i+1}$ and
one copy of each of the edges $e_{(i-1)(a-1)+1},\ldots, e_{i(a-1)}$;
set $E_{2k}$ contains one copy of path $P_{1}$ and
one copy of each of the edges $e_{(k-1)(a-1)+1},\ldots, e_{k(a-1)}$.
See Table~\ref{table-multiplicities} for an example with $k = 3$; the table contains
numbers of copies of $e_j$, resp., of $P_j$, in each set $E_i$.

\begin{table}[h!]
     \centering
     {
         \footnotesize
         \renewcommand{\arraystretch}{1}
         \setlength\tabcolsep{4pt}
         \begin{tabular}{|c||c|c|c||c|c|c|c|c||c|c|c||c|c|c|}
         \hline
         & {$P_1$} & {$P_2$} & {$P_3$} & {$e_1$} & {$e_2$} & {$e_3$} & {$\ldots$} & {$e_{a-1}$} & {$e_{a}$} & {$\ldots$} & {$e_{2(a-1)}$} & {$e_{2(a-1)+1}$} & {$\ldots$} & {$e_{3(a-1)}$}\\
         \hline
         {$E_1$} & {$1$} & & & {$2$} & & {$1$} & { $\ldots$} & {$1$} & & &  &  & & \\
         \hline
         {$E_2$} & & {$1$} & &  & {$2$} & {$1$} & { $\ldots$} & {$1$} & & &  &  & & \\
         \hline
         {$E_3$} & & {$1$} & &  & &  & & & {$1$} & { $\ldots$} &  {$1$} &  & & \\
         \hline
         {$E_4$} & &  & {$1$} &  & &  & & & {$1$} & { $\ldots$} &  {$1$} &  & & \\
         \hline
         {$E_5$} & & &{$1$} &  &  & & & &  &  & & {$1$} & { $\ldots$} & {$1$} \\
         \hline
         {$E_6$} & {$1$} & &&  &  & & & &  &  & & {$1$} & { $\ldots$} & {$1$} \\
         \hline
     \end{tabular}
     }
     \caption{Numbers specifying the construction of a locally connected Eulerian $(n^{2k})$-partition of $\HH$ in the proof of Proposition~\ref{prop:n2kak}, case $k>1$. Empty cells correspond to zero entries.}\label{table-multiplicities}
\end{table}

It can be verified that $\{E_1,\ldots, E_{2k}\}$ is a locally connected Eulerian $(n^{2k})$-partition of $\HH$.
Thus, the pair $(\bfd;\bft)$ is realizable by Theorem~\ref{thm:geographic-degree-sequences}.
\end{proof}

\section{Some infinite families of sequence pairs not realizable on the sphere}
\label{sec:non-realizable-on-sphere}

Consider a realizable pair $(\bfd;\bft)$. Then $\sum_{i = 1}^nd_i = \sum_{j = 1}^mt_j$ and the Euler characteristic of all surfaces of the maps generating a plan with degree sequence $(\bfd;\bft)$ is unambiguously determined by the degree sequence, according to equality $\chi=n-\ell+m$, where $2\ell = \sum_{i = 1}^nd_i = \sum_{j = 1}^mt_j$.
Therefore, from now on we will only consider pairs  $(\bfd;\bft)$ such that $\sum_{i = 1}^nd_i = \sum_{j = 1}^mt_j = 2\ell$ for some
integer $\ell$, and denote the value of $n-\ell+m$ by $\chi(\bfd;\bft)$.

\begin{lemma}\label{lem:connected}
Let $G$ be a connected graph and let $\mathcal{E} = \{E_1,\ldots, E_m\}$
be a locally connected partition of the double graph $\GG$. Then, there
is no vertex $v\in V(G)$ and a partition $\{F_1,F_2\}$ of the set of
edges of $G$ incident to $v$ such that $F_1,F_2\neq\emptyset$, both
copies of all edges in $F_1$ belong to $E_j$, and
no copy of any edge in $F_2$ belongs to $E_j$ for some $j\in \{1,\ldots,
m\}$.
\end{lemma}

\begin{proof}
Otherwise, vertex $j$ would be an isolated vertex in the graph
$H_v^{\mathcal{E}}$,
contrary to the local connectivity of the partition.
\end{proof}

\begin{proposition}\label{thm:chi=2-example-1}
Let $a>b\geq 1$ be integers such that $a+b$ is even, let $n = (a+b)/2$, and
consider the sequences $\bfd=(2,\ldots, 2)\in \ZZ_+^{n}$ and
$\bft=(a,b)\in \ZZ_+^{2}$.
Then the pair $(\bfd;\bft)$ is not realizable.
(Note that
$\chi(\bfd;\bft)=2$.)
\end{proposition}

\begin{proof}
Consider an arbitrary connected realization $G$ of the degree sequence
$\bfd$. This is a connected $2$-regular graph on $n$ vertices, that is,
an $n$-cycle. By Theorem~\ref{thm:geographic-degree-sequences}, it is
enough to show that
the double graph $\GG$ has no locally connected Eulerian $\bft$-partition.
Assume indirectly that there exists a locally connected Eulerian
partition $\mathcal{E} = \{E_1,E_2\}$ of $\GG$ such that $|E_1| = a$ and
$|E_2| = b$.
Since $a+b = 2n$ and $a>b$, we have $b<n$.
Lemma~\ref{lem:connected}
implies that each color class contains at least one copy of each edge of
$G$.
Since $b<n$, this is impossible.
\end{proof}

\begin{proposition}\label{thm:chi=2-example-2}
Let $n\geq 2$ be an integer, and consider the sequences
$\bfd=(3,2,\ldots, 2,1)\in \ZZ_+^{n}$ and
$\bft=(n,n)\in \ZZ_+^{2}$.
Then the pair $(\bfd;\bft)$ is not realizable.
(Note that
$\chi(\bfd;\bft)=2$.)
\end{proposition}

\begin{proof}
The arguments are similar as in the proof of Proposition~\ref{thm:chi=2-example-1}.
Consider an arbitrary connected realization $G$ of the degree sequence
$\bfd$.
This is a connected graph on $n$ vertices, in which vertex $v_1$ is of
degree $3$, vertex
$v_n$ is of degree $1$, and all other vertices are of degree $2$.
Thus such a graph consists of a path $P$ from $v_1$ to $v_n$ and a cycle
$C$ through $v_1$ such that $v_1$ is the only common vertex of $P$ and $C$.
Assume indirectly that there exists a locally connected Eulerian
partition $\mathcal{E} = \{E_1,E_2\}$ of $\GG$ such that $|E_1| = |E_2|
= n$.
We can assume that $v_n$ is a vertex of the Eulerian subgraph of $\GG$
corresponding to $E_1$. By Lemma~\ref{lem:connected}, $E_1$
must contain both copies of each edge of $P$, and at least one copy of
each edge on $C$. Since $G$ has exactly $n$ edges, the above implies
that $|E_1|\geq  |E(C)|+2|E(P)|  = n+|E(P)| > n$, a contradiction.
\end{proof}

\begin{proposition}\label{thm:chi=2-example-4}
Let $a,b\geq 2$ be integers and let $t_1\geq \ldots\geq t_a$ be positive
even integers such that $\sum_{i = 1}^{a}t_i = 2(a+b)$. Consider the
sequences $\bfd=(2a+b-1,1,\ldots,1)\in \ZZ_+^{b+2}$ and
$\bft=(t_1,\ldots, t_a)\in \ZZ_+^{a}$.
Then the pair $(\bfd;\bft)$ is not realizable.
(Note that
$\chi(\bfd;\bft)=2$.)
\end{proposition}

\begin{proof}
Consider an arbitrary connected realization $G$ of the degree sequence
$\bfd$.
This is a graph consisting of a vertex $v_1$ of degree $2a+b-1$ and
$b+1$ vertices of degree~$1$, all adjacent to $v_1$. Thus, there are
$a-1$ loops at $v_1$.
Assume indirectly that there exists a locally connected Eulerian
partition $\mathcal{E} = \{E_1,\ldots,E_a\}$ of $\GG$ such that
$|E_j| = t_j$ for all $j\in \{1,\ldots, a\}$. If for some $j\in
\{1,\ldots,m\}$ a copy of a pendant edge belongs to $E_j$, then both
copies must belong to it.
Thus, by Lemma~\ref{lem:connected}, every $E_j$ must contain at
least one of the loops. Since the sizes $t_j$ are all even, every $E_j$
must contain at least two loops in the double graph $\GG$. This requires
at least $2a$ loops in
$\GG$. However, since $\GG$ has only $2a-2$ loops, this is impossible.
\end{proof}

\begin{proposition}\label{thm:chi=2-example-3}
Let $a\geq 0$ and $b\geq 0$ be integers such that $b < a+3 \leq 2b$, and
consider the sequences $\bfd=(3,3,2,\ldots,2)\in \ZZ_+^{a+2}$ and
$\bft=(a+3,b, a+3-b)\in \ZZ_+^{3}$.
Then the pair $(\bfd;\bft)$ is not realizable.
(Note that
$\chi(\bfd;\bft)=2$.)
\end{proposition}

\begin{proof}
Consider an arbitrary connected realization $G = (V,E)$ of the degree sequence
$\bfd$. This is a graph on $a+2$ vertices, where the first two vertices, $v_1$ and $v_2$, have degrees $3$, and all other vertices are of degree $2$.
Thus such a graph consists either of
one path between $v_1$ and $v_2$, and one cycle at each of these vertices,
where the path and the two cycles are otherwise vertex disjoint
(see the left-hand side of Fig.~\ref{fig:graph-5-4} for an example) or of
three internally vertex-disjoint paths between $v_1$ and $v_2$
(see the right-hand side of Fig.~\ref{fig:graph-5-4} for an example).

\medskip
\begin{figure}[h!]
	\begin{center}
		\includegraphics[width=0.9\textwidth]{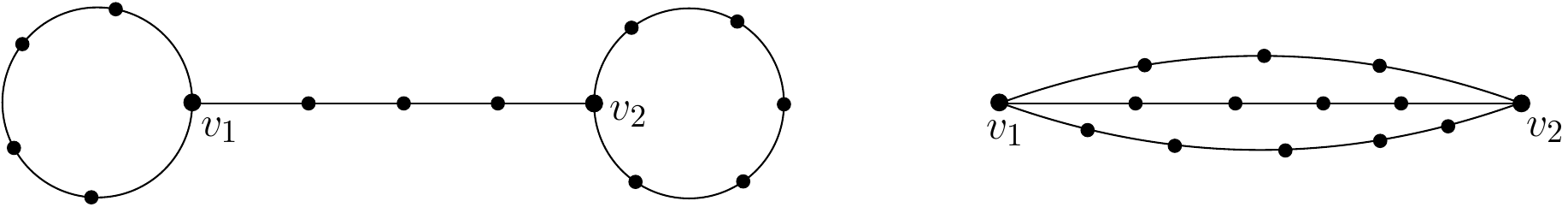}
	\end{center}
	\caption{The two types of connected realizations of the degree sequence $\bfd$.
		In the two examples, we have $a = 12$. The value of $b$ can be any integer in the set $\{8,\ldots, 14\}$.
	}\label{fig:graph-5-4}
\end{figure}

Let $\mathcal{W}$ denote the set consisting of the edge sets of the $v_1,v_2$-paths in $G$ and the cycles at $v_1$ and at $v_2$ (if any).

Assume indirectly that the double graph $\GG$ has a locally connected Eulerian $\bft$-partition $\mathcal{E} = \{E_1,E_2,E_3\}$.
Since  $\bft=(a+3,b, a+3-b)$, we have
\begin{equation}\label{e:size-new}
|E_1| = a+3, ~~~|E_2| = b, ~~~|E_3| = a+3-b\,.
\end{equation}
For each $j\in \{1,2,3\}$ and each edge $e\in E$, let $\mu_j(e)\in \{0,1,2\}$ denote the number of copies of $e$ in $\GG$ that belong to color class $E_j$.
	
Let us first note the following.

\begin{lemma}\label{l0-new}
		Let $W\in \mathcal{W}$ and let $e$ and $e'$ be two edges in $W$. Then
		\begin{equation}\label{e1-new}
		\mu_j(e)=\mu_j(e') ~~~\text{ for all }~~~ j=1,2,3\,.
		\end{equation}
\end{lemma}
		
\begin{proof}
It is enough to show the claim for any two adjacent and distinct edges $e$ and $e'$ in $W$ with a common endpoint $v$ such that $v\neq  v_1, v_2$. If for some $j\in \{1,\ldots,2b\}$ we have $\mu_j(e)\neq \mu_j(e')$, then either the fact that the graph $(V,E_j)$ defined by the color class $j$ is Eulerian,
or the connectedness of the graph $H_v^{\mathcal{E}}$ fails.
\end{proof}

To state the next lemma, we will need the following notation.
For all $j\in \{1,2,3\}$, we denote by $\textrm{supp}(E_{j})$ the {\it support of class $E_j$}, that is,
the set of edges $e$ in $G$ such that $\mu_j(e)\geq 1$.

\begin{lemma}\label{l0.5-new}
If the color class $E_j$ contains a copy of an edge from some $W\in \mathcal{W}$,
then $W$ is a subset of $\textrm{supp}(E_j)$.
\end{lemma}

\begin{proof}
Immediate from the previous lemma.
\end{proof}
	
We now analyze the two cases depending on the structure of $G$.

Suppose first that $G$ consists of one cycle $P$ at vertex $v_1$, one path $Q$ between $v_1$ and $v_2$,
and one cycle $R$ at vertex $v_2$, where the path and the two cycles are otherwise vertex disjoint
(as in the left-hand side of Fig.~\ref{fig:graph-5-4}). Let $e$ be the edge on path $Q$ incident with $v_1$ and let $j$ be a color class such that $e$ belongs to the support of $E_j$. Lemma~\ref{l0.5-new} implies that the set $\textrm{supp}(E_j)$ contains the edge set of $Q$ as a subset. By Lemma~\ref{l0-new}, all the edges of cycle $P$ appear in $E_j$ with the same multiplicity ($0$, $1$, or $2$). Since the color class $E_j$ is Eulerian, this implies that $\mu_j(e) = 2$. Using Lemma~\ref{l0-new},
we infer that all the edges of path $Q$ appear in $E_j$ with multiplicity $2$.
Lemma~\ref{lem:connected} implies that some edge of cycle $P$ incident with $v_1$ belongs to $\textrm{supp}(E_j)$ and consequently
the edge set of cycle $P$ is contained, as a subset, in $\textrm{supp}(E_j)$. Similarly, by considering vertex $v_2$, we obtain that the edge set of cycle $R$ is contained, as a subset, in $\textrm{supp}(E_j)$.
Consequently, $|E_{j}| \geq |E(P)|+2|E(Q)|+|E(R)|>|E| = a+3$, a contradiction.

\medskip
Suppose now that $G$ consists of three internally vertex-disjoint paths $Q_1$, $Q_2$, and $Q_3$ between $v_1$ and $v_2$
(as in the right-hand side of Fig.~\ref{fig:graph-5-4}).

\begin{lemma}\label{l2-new}
For any of the color classes $E_j$, the edge sets of at least two of the paths $Q_1$, $Q_2$, $Q_3$,
are contained, as subsets, in $\textrm{supp}(E_j)$.
\end{lemma}

\begin{proof}
Since $|E_j|>0$ for all $j\in \{1,2,3\}$, Lemma~\ref{l0.5-new} implies that for each color class $E_j$, its support $\textrm{supp}(E_j)$ contains, as a subset, the edge set of at least one $v_1,v_2$-path.
	
Assume indirectly that $\textrm{supp}(E_j)=E(Q_i)$ for some path $Q_i$.
By Lemma~\ref{l0-new} we have either $\mu_j(e)=1$ for all edges $e\in E(Q_i)$ or $\mu_j(e)=2$ for all edges $e \in E(Q_i)$.
We cannot have $\mu_j(e)=1$ for all edges $e\in E(Q_i)$, since otherwise the color class $E_j$ would not be Eulerian.
Thus, we have $\mu_j(e)=2$ for all edges $e\in E(Q_i)$. This, however, leads to a contradiction with Lemma~\ref{lem:connected} due to vertex $v_1$.
\end{proof}

For each color class $j\in \{1,2,3\}$, let $D_j = \sum_{e\ni v_1}\mu_j(e)$.
Since each color class is Eulerian, the numbers $D_j$ are all positive and even.
Furthermore,
$$\sum_{j= 1}^3D_j = \sum_{j= 1}^3\sum_{e\ni v_1}\mu_j(e) = \sum_{e\ni v_1}\sum_{j= 1}^3\mu_j(e) = 6\,,$$
where the last equality follows from the fact that
\textrm{$\mathcal{E}$} is a partition of the edge set of $\GG$.
It follows that $D_j = 2$ for all $j\in \{1,2,3\}$.
By Lemma~\ref{l2-new}, we have only one possibility up to renaming the paths:
for all $j\in \{1,2,3\}$, the support of color class $E_j$ contains the edges of
paths $Q_j$ and $Q_{j+1}$, with multiplicity equal to $1$ (indices modulo $3$).
For all $j\in \{1,2,3\}$, let us denote by $q_j$ the number of edges in path $Q_j$.
Since $|E_j| = q_j+q_{j+1}$ for all $j\in \{1,2,3\}$ (indiced modulo $3$), equations~\eqref{e:size-new} imply
\begin{eqnarray*}
q_1+q_2 &=& a+3\,,\\
q_2+q_3 &=& b\,,\\
q_3+q_1 &=& a+3-b\,.
\end{eqnarray*}
This system has a unique solution $(q_1,q_2,q_3) = (a+3-b,b,0)$. In particular, $q_3 = 0$, which contradicts the fact that all three paths have strictly positive length.
\end{proof}

\begin{table}[h!]
     \centering
     {
         \footnotesize
         \renewcommand{\arraystretch}{1.2}
         \begin{tabular}{|c|c|c|c|}
             \hline
             Proposition~\ref{thm:chi=2-example-1} &
             Proposition~\ref{thm:chi=2-example-2} &
             Proposition~\ref{thm:chi=2-example-4} &
             Proposition~\ref{thm:chi=2-example-3} \\
             \hline
             { $(a,b; 2^{(a+b)/2})$} &
             {  $(3,2^{n-2},1;n,n)$} &
             {  $(2a+b-1,1^{b+1}; t_1,\ldots,t_a)$} &
             {  $(3,3,2^a;a+3,b,a+3-b)$} \\
             \hline
             $(3,1; 2,2)$     & $(3,1; 2,2)$   & $(5,1,1,1;
4,4)$              & $(3,3; 3,2,1)$            \\
             $(4,2; 2,2,2)$     & $(3,2,1; 3,3)$   & $(5,1,1,1;
6,2)$              & $(3,3,2; 4,2,2)$          \\
             $(5,1; 2,2,2)$   & $(3,2,2,1; 4,4)$ & $(6,1,1,1,1;
6,4)$            & $(3,3,2; 4,3,1)$        \\
             $(5,3; 2,2,2,2)$   & $(3,2,2,2,1; 5,5)$ & $(6,1,1,1,1;
8,2)$            & $(3,3,2,2; 5,3,2)$      \\
             $(6,2; 2,2,2,2)$   & $(3,2,2,2,2,1; 6,6)$ & $(7,1,1,1;
4,4,2)$            & $(3,3,2,2; 5,4,1)$    \\
             $(7,1; 2,2,2,2)$  & $(3,2,2,2,2,2,1; 7,7)$ & $(7,1,1,1;
6,2,2)$            & $(3,3,2,2,2; 6,3,3)$  \\
             $(6,4; 2,2,2,2,2)$  &                  & $(8,1,1,1,1;
4,4,4)$          & $(3,3,2,2,2; 6,4,2)$                          \\
             $(7,3; 2,2,2,2,2)$  &                  & $(8,1,1,1,1;
6,4,2)$         & $(3,3,2,2,2; 6,5,1)$                         \\
             $(8,2; 2,2,2,2,2)$  &                  & $(8,1,1,1,1;
8,2,2)$          &                           \\
             $(9,1; 2,2,2,2,2)$& &           &                           \\

             \hline
         \end{tabular}
     }
     \caption{Examples following Propositions 5.1 --
5.4}\label{table-propositions}
\end{table}

\bigskip
We close this section with a list some further examples of
non-realizable pairs. We leave it to the reader to verify that they are indeed non-realizable.

\begin{proposition}\label{example-1}
Let $a\geq 1$ and $\alpha\geq \beta\geq 1$ be integers such that
$\alpha+\beta\neq a+2$ and
$\alpha+\beta/2\le a+2$. Then, the pair $(2a+4-(\alpha+\beta), \alpha,
\beta; 4, 2^a)$ is not realizable.
\end{proposition}

\begin{proposition}\label{example-2}
Let $a\geq 1$ and $\alpha\geq  \beta\geq  \gamma\geq  \delta$ be
integers such that
  $\alpha+\beta+\gamma+\delta = 2a+6$, and ($\alpha+\delta \neq a+3$ or
$\gamma+\delta\geq a+3$) and $\alpha\neq a+3$.
       Then, the pair
     $(\alpha, \beta, \gamma, \delta; 6, 2^a)$ is not realizable.
\end{proposition}

\begin{proposition}\label{example-3}
Let $a\geq 0$ and $\alpha\geq  \beta\geq  \gamma\geq  \delta$ be
integers such that $\alpha+\beta+\gamma+\delta = 2a+8$ and\begin{itemize}
     \item $\alpha>a+2$ or $\delta= 1$ or $\alpha+\delta \neq a+4$,
     \item $\alpha\neq a+4$,
     \item $\alpha \neq a+3$ or $\gamma\neq 1$.
   \end{itemize}
       Then, the pair
       $(\alpha, \beta, \gamma, \delta; 4, 4, 2^a)$  is not realizable.
\end{proposition}

\begin{proposition}\label{example-4}
For every integer $a\geq  0$, the pair $(a+3,a+3,1,1;5,3,2^a)$ is not
realizable.
\end{proposition}

\begin{proposition}\label{example-5}
For every integer $a\geq  1$ and every three even integers $\alpha\geq
\beta\geq  \gamma$ such that $\alpha+\beta+\gamma = 2a + 6$, the pair
$(\alpha,\beta,\gamma; a+4,2,1^a)$ is not realizable.
\end{proposition}

In Tables~\ref{table-examples} and~\ref{table-examples2} we list some small examples of
non-realizable pairs given by Propositions~\ref{example-1}--\ref{example-5}.

\begin{table}[h!]
	\tiny
\centering
{
  \footnotesize
\renewcommand{\arraystretch}{1.2}
\begin{tabular}{|c|c|c|}
   \hline
Proposition~\ref{example-1} &
Proposition~\ref{example-2} &
Proposition~\ref{example-3} \\
   \hline
{ $(2a+4-(\alpha+\beta),\alpha,\beta;4,2^a)$} &
{  $(\alpha, \beta, \gamma, \delta; 6, 2^a)$} &
{  $(\alpha, \beta, \gamma, \delta; 4, 4, 2^a)$} \\
  \hline
   $(2,2,2; 4,2)$     &   $(2,2,2,2; 6,2)$   &   $(3,2,2,1;
4,4)$              \\
  $(4,1,1; 4,2)$     & $(5,1,1,1; 6,2)$   & $(5,1,1,1;
4,4)$              \\
  $(3,3,2; 4,2,2)$   & $(3,3,3,1; 6;2,2)$ & $(3,3,3,1;
4,4,2)$            \\
  $(5,2,1; 4,2,2)$   & $(4,2,2,2; 6;2,2)$ & $(4,2,2,2;
4,4,2)$            \\
  $(6,1,1; 4,2,2)$   & $(6,2,1,1; 6,2,2)$ & $(6,2,1,1;
4,4,2)$            \\
$(4,3,3; 4,2,2,2)$  & $(7,1,1,1; 6,2,2)$ & $(7,1,1,1; 4,4,2)$
\\
$(4,4,2; 4,2,2,2)$  &                  & $(4,4,3,1; 4,4,2,2)$
\\
$(6,2,2; 4,2,2,2)$  &                  & $(5,3,2,2; 4,4,2,2)$
\\
$(6,3,1; 4,2,2,2)$  &                  & $(7,2,2,1; 4,4,2,2)$
\\
$(7,2,1; 4,2,2,2)$  &                  & $(7,3,1,1; 4,4,2,2)$
\\
$(8,1,1; 4,2,2,2)$  &                  & $(8,2,1,1; 4,4,2,2)$\\
                     &                  & $(9,1,1,1; 4,4,2,2)$\\

   \hline
\end{tabular}
}
\caption{Small examples of non-realizable pairs given by
Examples~\ref{example-1}--\ref{example-3}, $\chi =
2$.}\label{table-examples}
\end{table}

\begin{table}[h!]
	\tiny
	\centering
	{
		\footnotesize
		\renewcommand{\arraystretch}{1.2}
		\begin{tabular}{|c|c|}
			\hline
			Proposition~\ref{example-4} &
			Proposition~\ref{example-5} \\
			\hline
			{  $(a+3,a+3,1,1;5,3,2^a)$} &
			{  $(\alpha,\beta,\gamma; a+4,2,1^a)$}\\
			\hline
 $(3,3,1,1; 5,3)$   & $(2,2,2; 4,2)$ \\
 $(4,4,1,1; 5,3,2)$ & $(4,2,2; 5,2,1)$ \\
 $(5,5,1,1; 5,3,2,2)$        & $(4,4,2; 6,2,1,1)$  \\
 $(6,6,1,1; 5,3,2,2,2)$      & $(6,2,2; 6,2,1,1)$ \\
 $(7,7,1,1; 5,3,2,2,2,2)$    & $(4,4,4; 7,2,1,1,1)$ \\
 $(8,8,1,1; 5,3,2,2,2,2,2)$  & $(6,4,2; 7,2,1,1,1)$\\
                             & $(8,2,2; 7,2,1,1,1)$ \\
                             & $(6,4,4; 8,2,1,1,1,1)$\\
                             & $(6,6,2; 8,2,1,1,1,1)$\\
                             & $(8,4,2; 8,2,1,1,1,1)$ \\			
			\hline
		\end{tabular}
	}
	\caption{(More) small examples of non-realizable pairs given by Examples~\ref{example-4} and \ref{example-5}, $\chi = 2$.}
	\label{table-examples2}
\end{table}

\section{A two-parametric infinite family of sequences pairs not realizable on the projective plane}
\label{sec:non-realizable-on-projective-plane}

Recall that, given two integer sequences $\bfd=(d_1,\ldots,d_n)\in\ZZ_+^n$ and
$\bft=(t_1,\ldots,t_m)\in\ZZ_+^m$, the pair $(\bfd;\bft)$ is said to be realizable if it is the degree sequence of a geographic plan.

\begin{theorem}\label{thm:chi=1}
Let $a\geq 3$ and $b\geq 1$ be integers, and consider the sequences $\bfd=(a,a,\ldots,a)\in \ZZ_+^{2b}$ and $\bft=(2b+1,2b+1,2,\ldots,2)\in \ZZ_+^{ab-2b+1}$. Then the pair $(\bfd;\bft)$ is not realizable.
(Note that
$\chi(\bfd;\bft)=1$.)
\end{theorem}

\begin{proof}
By Theorem~\ref{thm:geographic-degree-sequences}, it suffices to show that for every connected realization $G$ of $\bft$ the double graph $\GG$ does not have any locally connected Eulerian \hbox{$\bfd$-partition}. We follow the proof strategy of Proposition~\ref{thm:chi=2-example-3}, except that this time we are proving non-realizability, and the family of possible connected realizations is more general.

Consider an arbitrary connected realization $G=(V,E)$ of the degree sequence $\bft$. This is a graph on $ab-2b+1$ vertices, where the first two vertices, $v_1$ and $v_2$, have degrees $2b+1$, and all other vertices are of degree $2$. Thus such a graph consists of an odd number, say $2c+1$, of paths between $v_1$ and $v_2$, and some cycles at each of these vertices. These paths and cycles are otherwise vertex disjoint. Furthermore, since the degrees of $v_1$ and $v_2$ are both $2b+1$, the number of cycles at $v_1$ and $v_2$ are both equal to $b-c$. Thus we must have $0\leq c\leq b$. See Fig.~\ref{fig:graph-6-1} for an example of such a graph.

\medskip
\begin{figure}[h!]
	\begin{center}
		\includegraphics[width=0.7\textwidth]{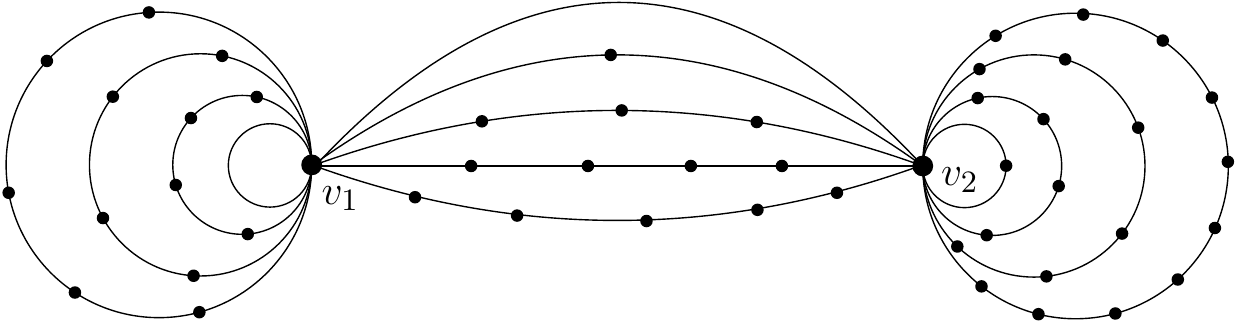}
	\end{center}
	\caption{An example of a connected realization of the degree sequence $\bft$, with
		$a = 10$, $b = 6$, and $c = 2$.}\label{fig:graph-6-1}
\end{figure}

Let us denote the cycles at $v_1$ by $P_1,\ldots, P_{b-c}$,
the $v_1$,$v_2$-paths by $Q_1,\ldots, Q_{2c+1}$, and the cycles at $v_2$ by $R_{1},\ldots,R_{b-c}$.
Furthermore, let
$$\mathcal{W} = \{P_1,\ldots, P_{b-c},Q_1,\ldots, Q_{2c+1},R_{1},\ldots,R_{b-c}\}\,.$$

Assume indirectly that the double graph $\GG$ has a locally connected Eulerian $\bfd$-partition
$\mathcal{E} = \{E_1,\ldots, E_{2b}\}$. Then
\begin{equation}\label{e:color-partition}
\textrm{$\mathcal{E}$
is a partition of the edge set of $\GG$.}\end{equation}
Furthermore, since  $\bfd=(a,a,\ldots,a)\in \ZZ_+^{2b}$, we have
\begin{equation}\label{e:size}
|E_j| = a ~~~\text{ for all }j\in \{1,\ldots, 2b\}\,.
\end{equation}
edge $e\in E$, let us denote by $\mu_j(e)$ the size of the intersection
For each $j\in \{1,\ldots, 2b\}$ and each edge $e\in E$, let $\mu_j(e)\in \{0,1,2\}$ denote the number of copies of $e$ in $\GG$ that belong to color class $E_j$.

Applying the same arguments as in the proof of Lemma~\ref{l0-new} (which corresponds to the case $b = 1$)
leads to the following.

\begin{lemma}\label{l0}
Let $W\in \mathcal{W}$ and let $e$ and $e'$ be two edges in $W$. Then
\begin{equation}\label{e1}
\mu_j(e)=\mu_j(e') ~~~\text{ for all }~~~ j=1,\ldots,2b\,.
\end{equation}
\end{lemma}

To state the next lemma, we will need the following notation.
For all $j\in \{1,\ldots, 2b\}$, we denote by $\textrm{supp}(E_j)$ the {\it support of class $E_j$}, that is,
the set of edges $e$ in $G$ such that $\mu_j(e)\geq 1$.

\begin{lemma}\label{l0.5}
If the color class $E_j$ contains a copy of an edge from some $W\in \mathcal{W}$,
then $W$ is a subset of
$\textrm{supp}(E_j)$.
\end{lemma}

\begin{proof}
Immediate from the previous lemma.
\end{proof}

\begin{lemma}\label{l1}
For any of the color classes $E_j$
at least two different paths/cycles are in $\textrm{supp}(E_j)$.
\end{lemma}

\begin{proof}
	Since $|E_j| =a>0$ for all $j\in \{1,\ldots, 2b\}$, Lemma~\ref{l0.5} implies that for each color class $E_j$, the set $\textrm{supp}(E_j)$ contains, as a subset, at least one member of $\mathcal{W}$.

	Assume indirectly that $\textrm{supp}(E_j)=W$ for some $W\in \mathcal{W}$. By symmetry, we may assume that vertex $v_1$ is incident with an edge in $W$. By \eqref{e1} we have either $\mu_j(e)=1$ for all edges $e\in W$ or $\mu_j(e)=2$ for all edges $e \in W$.
	
	In the first case this implies $|W| = |\textrm{supp}(E_j)| = |E_j| = a$, and thus, since $\mathcal{E}$ is a partition of the edges of $\GG$, conditions~\eqref{e:size} and~\eqref{e1} imply that we must have another color $j'$ such that $\textrm{supp}(E_{j'})=W$. Since the degree of $v_1$ in $G$ is $2b+1\geq 3$, there is an edge incident with $v_1$ that does not belong to $\textrm{supp}(E_j)\cup \textrm{supp}(E_{j'}) = W$. Thus colors $j$ and $j'$ would not be connected in $H^{\mathcal{E}}_{v_1}$ to the other colors at vertex $v_1$.
	
	In the latter case, $a=2|W|$, and color $j$ is not connected to the other colors at vertex $v_1$.
	
	In both cases we get a contradiction with the connectedness of the graph  $H^{\mathcal{E}}_{v_1}$.
\end{proof}

\begin{lemma}\label{l1.5}
We have $c = b$, or, in other words, there are no cycles $P_i$ or $R_j$.
\end{lemma}

\begin{proof}
Recall that $c\le b$. Assume that $c<b$, and consider all color classes that contain some of the paths between $v_1$ and $v_2$.
We claim that we can have at most $2c+1$ such color classes. By Lemma~\ref{l0.5} every color class is the union of copies of paths and cycles (from $\mathcal{W}$), hence, by the Eulerian property every color class must contain an even number of path copies.
So, those that do contain some, must contain at least 2. Since we have only $2(2c+1)$ such copies,
condition~\eqref{e:color-partition} implies we cannot have more than $2c+1$ such color classes.

It follows that we have at least $2b-(2c+1)=2(b-c)-1$ color classes that contain only cycles. Since each such color class $E_j$ is Eulerian, its support is connected, and hence contains only cycles incident with one of the two vertices $v_1$ and $v_2$.
Thus at least one of these vertices, say $v_1$, is incident with at least $(b-c)$ color classes that involve only cycles through this vertex. Thus by the fact that we have exactly $(b-c)$ cycles through vertex $v_1$, by Lemma~\ref{l1}, and by the fact that the graph $H_{v_1}^{\mathcal E}$ is connected, we must have exactly $(b-c)$ such color classes, each containing exactly two cycles (with multiplicities one). This implies that these $(b-c)$ color classes cover all edges that are in the cycles incident with vertex $v_1$. Since this vertex has some additional edges in the odd number of paths connecting it to vertex $v_2$, those are also covered by some color classes, which however cannot be connected in the graph  $H_{v_1}^{\mathcal E}$ to the $(b-c)$ color classes covering the cycles. This contradicts the connectedness of $H_{v_1}^{\mathcal E}$.
\end{proof}

Since $c=b$, the graph consists of $2b+1$ edge-disjoint paths connecting vertices $v_1$ and $v_2$.
For each $j \in \{1,\ldots, 2b\}$, let $D_j = \sum_{e\ni v_1}\mu_j(e)$.
Since each color class is Eulerian, the numbers $D_j$ are all positive and even.
Furthermore,
$$\sum_{j = 1}^{2b}D_j =
\sum_{j = 1}^{2b}\sum_{e\ni v_1}\mu_j(e) =
\sum_{e\ni v_1}\sum_{j = 1}^{2b}\mu_j(e) = 2(2b+1)\,,$$
where the last equality follows from condition~\eqref{e:color-partition}.
Consequently, $D_j = 2$ for all but one of the color classes, say $j = 1$, for which
$D_1 = 4$. By Lemma~\ref{l1}, the support of each color class $E_j$ such that $D_j = 2$
contains exactly two of the $2b+1$ paths and $\mu_j(e) = 1$ for each edge $e\in \textrm{supp}(E_j)$.
For $j = 1$, we have three possible cases up to renaming the paths:
\begin{enumerate}[(a)]
  \item The support of color class $E_1$ contains two paths both with multiplicities equal to $2$.

  \item The support of color class $E_1$ contains path $Q_1$ with multiplicity $2$ and paths $Q_2$ and $Q_3$ with multiplicity $1$.

  \item The support of color class $E_1$ contains exactly $4$ of the $2b+1$ paths, say
  $Q_1$, $Q_2$, $Q_3$, and $Q_4$, each with multiplicity $1$.
\end{enumerate}

Case (a) is not possible, as color $1$ is not connected in $H^{\mathcal{E}}_{v_1}$ to the other colors, contradicting the
connectedness of $H^{\mathcal{E}}_{v_1}$. (Note that the total number of color classes is $2b>1$.)

In case (b) equality~\eqref{e:size} implies the following equality:
\begin{equation}\label{e3}
2|E(Q_1)|+|E(Q_2)|+|E(Q_3)|=a.
\end{equation}
Let us create an auxiliary graph $\Gamma$ with vertex set $\{1,\ldots, 2b+1\}$ (the indices of the paths)
such that for every color $j\neq 1$ such that the support  of $E_j$
contains exactly $Q_k$ and $Q_\ell$ we connect $k$ and $\ell$ with an edge.
Graph $\Gamma$ has $2b+1$ vertices and $2b-1$ edges;
vertex $1$ has degree $0$, vertices $2$ and $3$ have degree $1$, and all other vertices have degree $2$.
Since $2$ and $3$ are the only vertices of odd degree in $\Gamma$, they are connected by a path $\Pi$. We claim that $\Pi$ contains all $2b-1$ edges of $\Gamma$. Otherwise, the set of colors corresponding to the edges of $\Pi$, together with color $1$, are disconnected from the rest of the colors in $H^{\mathcal{E}}_{v_1}$, contradicting the
connectedness of $H^{\mathcal{E}}_{v_1}$. Since for every edge $\{k,\ell\}$ in $\Gamma$ we have $|E(Q_k)|+|E(Q_\ell)| = a$ by equality~\eqref{e:size} and $\Pi$ contains an odd number of edges, it follows that $|E(Q_2)|+|E(Q_3)|=a$. Thus, $|E(Q_1)|=0$ by \eqref{e3}, a contradiction.

Finally, consider case (c). We must have
\begin{equation}\label{e2}
|E(Q_1)|+|E(Q_2)|+|E(Q_3)|+|E(Q_4)|=a.
\end{equation}
Let us again consider the auxiliary graph $\Gamma$ as defined above.
Then in this graph every vertex has degree $2$, except vertices $1$, $2$, $3$, and $4$, which have degree equal to $1$. Thus the graphs consist of two paths, say one connecting $1$ to $2$, and the other connecting $3$ to $4$. Similarly as above, these two paths contain all $2b-1$ edges of $\Gamma$. Otherwise, the set of colors corresponding to the edges on these two paths, together with color $1$, are disconnected from the rest of the colors  in $H^{\mathcal{E}}_{v_1}$, contradicting the
connectedness of $H^{\mathcal{E}}_{v_1}$. Since $\Gamma$ has $2b-1$ edges, one of these paths, say the path connecting $1$ to $2$, has an odd number of edges. Thus, similarly as in case (b), we infer that $|E(Q_1)|+|E(Q_2)|=a$. By \eqref{e2} we get $|E(Q_3)|=|E(Q_4)|=0$, a contradiction.
\end{proof}

The non-realizable examples given by Theorem~\ref{thm:chi=1} are not the only non-realizable pairs for the projective plane. Our computer search of small non-realizable pairs (see Algorithm~\ref{algo} in Appendix) revealed that the two degree pairs $(3,3,3,3;7,4,1)$ and $(3,3,3,3;5,4,3)$ (both of which have $\chi = 1$) are also not realizable.

\section{Directions for future research}

\subsection{Realizability for $\chi\le 0$}

A pair  $(\bfd;\bft)$ of integer sequences $\bfd=(d_1,\ldots,d_n)\in\ZZ_+^n$ and $\bft=(t_1,\ldots,t_m)\in\ZZ_+^m$ is said to be \emph{feasible} if  $\sum_{i = 1}^nd_i = \sum_{j = 1}^mt_j$.
Recall that a pair $(\bfd;\bft)$ is said to be realizable if it is the degree sequence of a geographic plan.
While every realizable pair is feasible, we gave in Sections~\ref{sec:non-realizable-on-sphere} and~\ref{sec:non-realizable-on-projective-plane} several families of examples of feasible pairs that are not realizable.
Note that all those exceptions are about the sphere, $\chi = 2$, and the projective plane, $\chi = 1$.
The following conjecture is about $\chi\le 0$.

\begin{conjecture}\label{conj:realizability}
Every feasible pair $(\bfd;\bft)$ with $\chi(\bfd;\bft)\leq 0$ is realizable.
\end{conjecture}

In other words, Conjecture~\ref{conj:realizability} states that for every two integer sequences $\bfd=(d_1,\ldots,d_n)\in\ZZ_+^n$ and $\bft=(t_1,\ldots,t_m)\in\ZZ_+^m$ such that
$\sum_{i = 1}^nd_i = \sum_{j = 1}^mt_j \geq 2(n+m)$, pair $(\bfd;\bft)$
is the degree sequence of a geographic plan.
Using Theorem~\ref{thm:geographic-degree-sequences}, it is not difficult to obtain the following partial result in the support of Conjecture~\ref{conj:realizability}: for every two integer sequences $\bfd$ and $\bft$ as above in which
all entries are at least $10$,
pair $(\bfd;\bft)$ is realizable.

Given a sequence pair $(\bfd;\bft)$ with
$\bfd=(d_1,\ldots,d_n)\in\ZZ_+^n$ and $\bft=(t_1,\ldots,t_m)\in\ZZ_+^m$,
our technique does not distinguish between realizability of $(\bfd;\bft)$ on orientable
or non-orientable surface of the same Euler characteristic $\chi(\bfd;\bft)$; in particular,
it does not distinguish between the torus and the Klein bottle.
Using a different approach, Nikolai Adrianov obtained
the following result in 2017 (private communications).

\begin{sloppypar}
\begin{proposition}\label{prop:Klein-not-torus}
The following pairs of integer sequences are not realizable on the torus:
$(n,n;5,3,2^{n-4})$ for all integers $n\geq 4$, and
$(n,n,n,n;7,5,2^{2n-6})$ for all integers $n\geq 3$.
\end{proposition}
\end{sloppypar}

By Propositions~\ref{prop:nn532n-4} and~\ref{prop:nnnn752n-6}, pairs $(n,n;5,3,2^{n-4})$ are realizable for all integers $n\geq 4$, and pairs $(n,n,n,n;7,5,2^{2n-6})$ are realizable for all integers $n\geq 3$.
Since they are not realizable on the torus, they must be realizable on the Klein bottle.
On the other hand, we are not aware of any degree sequences of geographic plans
realizable on the torus but not on the Klein bottle.
More generally, we pose the following.

\begin{conjecture}\label{conj:non-orientable-realizability}
	Every feasible pair $(\bfd;\bft)$ with $\chi(\bfd;\bft)\leq 0$
	can be realized on a non-orientable surface.
\end{conjecture}

\subsection{A generalization: triangulation of surfaces}

Consider a surface $S$ and an arbitrary map $M$ on $S$ generated
by an embedding of a graph  $G = (V,E)$  into  $S$.
We keep the notation:
$V = \{v_1, \ldots, v_n\}$, $E = \{e_1, \ldots, e_\ell\}$, and
the set of countries is  $F = \{f_1, \ldots, f_m\}$.
Furthermore, $2\ell = \sum_{i=1}^n d_i = \sum_{j=1}^m t_j$, where
$d_i$  and  $t_j$  denote the degrees (that is, the numbers of neighboring vertices and countries)
for $v_i$  and  $f_j$, respectively.
Recall that  $\chi(S) = n - \ell + m$, by Euler's formula.
	
	\smallskip
In each country  $f \in F$  on $S$  let us fix a point
$f'$ and call it the {\em capital} of  $f$.
In the interior of each edge  $e \in E$  let us fix a point $e'$ and call it the {\em checkpoint} of  $e$.
The sets  of all capitals and checkpoints are denoted by
$F'$  and  $E'$  and colored red and green, respectively,
while vertices  of $V$ are colored blue.
Furthermore let us introduce	
	\begin{itemize}
		\item{} a green edge $(v, f')$
		whenever vertex  $v$  belongs to
		the  country  $f$  with capital $f'$;
		\item{} a blue edge $(e', f')$  whenever
		the corresponding country  $f$  and edge  $e$  are incident;
		\item{} a red edge $(v, e')$
		whenever vertex  $v$  is incident in the graph  $G = (V,E)$
		to the edge  $e$  corresponding to  $e'$.
	\end{itemize}
Thus, we  obtain a $3$-colored triangulation  $T$  of  $S$
on the vertex set  $V \cup E' \cup  F'$.
Such triangulations were introduced by Shabat and Voevodsky in 1990~\cite{MR1106916};
see figures on pages 203 and 209.
(Note that the dual maps appeared already in 1984 in a preprint by Grothendieck,
which was published in 1997~\cite{MR1483107}.)
	
\medskip	
Note that, by construction,  no vertex and edge of the same color are incident in  $T$.
In other words, the  edges incident to a  vertex of a given color
are colored with two remaining colors, and, by construction, these two colors
alternate on  $S$.  Hence, each vertex of  $T$  has even degree.
Let  $d'_i$, $t'_j,$  and $\delta_k$  denote the degrees in  $T$
of the vertices  $v_i \in V, f'_j \in F'$, and  $e'_k \in E'$  divided by  $2$, respectively.
The above construction implies
for all  $i \in \{1, \ldots, n\}$, $j \in\{1, \ldots, m\}$, $k \in \{1, \ldots, \ell\}$ we have equalities
$d'_i = d_i$, $t'_j = t_j$, and  $\delta_k = 2$, since the degree of
a checkpoint in $T$ always equals $4$. Hence,
	\begin{equation}
	\label{eq-mnl}
	\sum_{i=1}^n d'_i = \sum_{j=1}^m t'_j = \sum_{k=1}^\ell \delta_k = 2\ell\,.
	\end{equation}	
In \cite{MR3792528} Shabat noticed that one can waive the set of constraints $\delta_k = 2$ for all $k \in \{1, \ldots, \ell\}$ and consider
arbitrary $3$-colored triangulations satisfying all remaining properties considered above.
More precisely, let  $T$  be a triangulation  of  $S$ defined on the vertex set  $V \cup E' \cup  F'$ colorable by $3$ colors.
Furthermore, each node of $T$ is of even degree.
Characterize the trivectors  $({\bf d'}$, ${\bf t'}$, $\boldsymbol{\delta})$ of such triangulations.
In the case  when $\boldsymbol{\delta} = (2, \ldots, 2)$, this problem is reduced to characterizing
pairs $({\bf d}; {\bf t})$ realized as degree sequences of maps and considered in the present paper.
	
As an example, we can mention a result by Adrianov~(private communications) stating that the trivector
$({\bf d'}$, ${\bf t'}$, $\boldsymbol{\delta})$ satisfying~\eqref{eq-mnl}
with ${\bf d'}= {\bf t'} = (3, \ldots, 3)$, and
$\boldsymbol{\delta} = (4,3, \ldots, 3,2)$ cannot be realized on the torus.

\subsection*{Acknowledgements}

The work for this paper was done in the framework of bilateral projects between Slovenia and the USA and between Slovenia and the Russian federation, partially financed by the Slovenian Research Agency (BI-US/$16$--$17$--$030$, BI-US/$18$--$19$--$029$, and BI-RU/$19$--$20$--$022$). The second author gratefully acknowledges the  partial support of the Russian Science Foundation, grant 20-11-20203; this research topic was included in the HSE University Basic Research Program;
he is also grateful to Nikolai Adrianov and George Shabat for helpful discussions. The work of the third author is supported in part by the Slovenian Research Agency (I0-0035, research program P1-0285 and research projects J1-9110, N1-0102, and N1-0160). The third and fourth author gratefully acknowledge the European Commission for funding the InnoRenew CoE project (Grant Agreement \#739574) under the Horizon2020 Widespread-Teaming program and and the Republic of Slovenia (Investment funding of the Republic of Slovenia and the European Union of the European regional Development Fund).

\appendix

\newpage
\section*{Appendix: Computer search of small non-realizable pairs}

\begin{algorithm}[h!]
	\caption{Search for all non-realizable bi-vectors.}\label{algo}
	\KwIn{$\ell$ -- number of edges}
	\KwOut{{\it nonRealizableBivectors} -- the set of non-realizable bi-vectors.}
	\tcp{Notation: $n$ is the number of vertices, $m$ is the number of countries, $B_{G}$ is the original graph represented as a matrix in $\{0,1,2\}^{\ell \times n}$ (0 - no edge at vertex $v$, 1 - edge with endpoint $v$, 2 - loop at $v$), $B_H$ is the dual graph represented as a matrix in $\{0,1,2\}^{\ell \times m}$}
	${\it bivectors} \gets \emptyset$ \nllabel{l1}\\
	${\it allPossibleBivectors} \gets \emptyset$ \nllabel{l2}\\
	\For{$n = 1,\dots, \lfloor \ell/2\rfloor + 1$\nllabel{l3}}{
		${\it lines} \gets$ the set of all vectors in $\{0,1,2\}^n$ with sum of all elements equal to $2$\nllabel{l4} \\
		$\mathcal{B}_G \gets$ the set of all multisubsets of ${\it lines}$ of cardinality $\ell$ \nllabel{l5}\\
		\tcp{each such multisubset represents one matrix $B_G$, by taking the vectors in the multiset to be the rows of $B_G$.}
		$\mathcal{B}_G \gets$ remove from $\mathcal{B}_G$ all matrices $B_G$ representing disconnected graphs (using adjacency lists and BFS) \nllabel{l6}\\
		\ForEach{$B_G \in \mathcal{B}_G$\nllabel{l7}}{
            \For{$i = 1,\dots, n$\nllabel{l71}}{
                $d_i \gets$ sum of the entries in the $i$-th column of $B_G$\nllabel{l72}\\
            }
            ${\it columns} \gets$ the set of all vectors in $\{0,1,2\}^\ell$ and having even scalar product with each column of $B_G$ \nllabel{l8}\\
			\For{$m = n,\dots, \ell+n-2$\nllabel{l9}}{
				${\it temporarySet} \gets$ the set of all pairs $(\bfd;\bft)$ where $\bfd=(d_1,\ldots,d_n)\in\ZZ_+^n$, $\bft=(t_1,\ldots,t_m)\in\ZZ_+^m$, and the sum of all elements in $\bfd$ as well as in $\bft$ is $2\ell$ \nllabel{l15}\\
				${\it allPossibleBivectors} \gets {\it allPossibleBivectors} \cup {\it temporarySet}$\nllabel{lnew}\\
				$\mathcal{B}_H \gets$ the set of all multisubsets of ${\it columns}$ of cardinality $m$ such that all row sums of the corresponding matrix equal to $2$ \nllabel{l10}\\
				\tcp{each such multisubset represents one matrix $B_H$, by taking the vectors in the multiset to be the columns of $B_H$}
				$\mathcal{B}_H \gets$ remove from $\mathcal{B}_H$ all matrices $B_H$ representing disconnected graphs (using adjacency lists and BFS)\nllabel{l11}\\
				\ForEach{$B_H\in \mathcal{B}_H$\nllabel{l12}}{
					${\it bivector} \gets (d_1,\ldots, d_n;t_1,\ldots, t_m)$ where $t_j$ is the sum of the entries in the $j$-th column of $B_H$ for all $j\in \{1,\ldots, m\}$ \nllabel{l13}\\
					${\it bivectors} \gets {\it bivectors}\cup \{{\it bivector}\}$ \nllabel{l14}\\
				}
			}
		}
	}
	${\it nonRealizableBivectors} \gets {\it allPossibleBivectors} \setminus {\it bivectors}$  \nllabel{l16}\\
	\Return{\it nonRealizableBivectors}\; \nllabel{l17}
\end{algorithm}

The source code and the dataset of realizable degree sequences of dual graphs on surfaces are available at Zenodo\footnote{\url{http://doi.org/10.5281/zenodo.3833674}}.

The time complexity of the algorithm can be analyzed as follows.
Lines~\ref{l1}--\ref{l2} take $\mathcal{O}\left(1\right)$ time.
The for loop extending over lines~\ref{l3}--\ref{l14}
has $\mathcal{O}\left(\ell\right)$ iterations.
Each iteration takes the following time:
\begin{itemize}
  \item Line~\ref{l4} takes time $\mathcal{O}\left({n+1\choose 2} n\right)$, since there are
${n+1\choose 2}$ vectEors in the set ${\it lines}$, each of length $n$.
\item Line~\ref{l5} takes time $\mathcal{O}\left({{n+1\choose 2}+\ell-1\choose  \ell} \ell n\right)$, since there are
${{n+1\choose 2}+\ell-1\choose  \ell}$ multisubsets in $\mathcal{B}_G$, each of which can be stored in $\mathcal{O}\left(\ell n\right)$ space.
\item The connectedness check in line~\ref{l6} can be done for each $B_G$ in time $\mathcal{O}\left(\ell n\right)$, hence
the time complexity of line~\ref{l6} is altogether $\mathcal{O}\left({{n+1\choose 2}+\ell-1\choose  \ell} \ell n\right)$.
\item The for loop extending over lines~\ref{l7}--\ref{l14}
has $\mathcal{O}\left(\binom{{n+1\choose 2}+\ell-1}{\ell}\right)$ iterations.
Each iteration takes the following time:
 \begin{itemize}
 \item Lines~\ref{l71}--\ref{l72} take time $\mathcal{O}\left(\ell n\right)$.
 \item Line~\ref{l8} takes time $\mathcal{O}\left(3^\ell\ell n\right)$, since there are
$3^\ell$ vectors in the set ${\it columns}$ and for each of them we need $\mathcal{O}\left(\ell n\right)$ time to test the scalar product condition.
   \item The for loop extending over lines~\ref{l9}--\ref{l14}
has $\mathcal{O}\left(\ell\right)$ iterations. Each iteration takes the following time:
\begin{itemize}
  \item Lines~\ref{l15}--\ref{lnew} take in total $\mathcal{O}\left({n+2\ell-1\choose 2\ell}{m+2\ell-1\choose 2\ell}(n+m)\right)$,
since there are ${n+2\ell-1\choose 2\ell}$ possible vectors $\bfd$,
there are ${m+2\ell-1\choose 2\ell}$ possible vectors $\bft$, and each
bivector $(\bfd;\bft)$ is of length $n+m$.
  \item \sloppypar{Lines~\ref{l10}--\ref{l11} take time $\mathcal{O}\left({3^\ell+m-1\choose m} \ell m\right)$, since there are
${3^\ell+m-1\choose m}$ multisubsets in $\mathcal{B}_H$, each of which can be stored in $\mathcal{O}\left(\ell m\right)$ space, and the row sum and the connectedness check takes time $\mathcal{O}\left(\ell m\right)$ for each of them.}
  \item The for loop extending over lines~\ref{l12}--\ref{l14} has $\mathcal{O}\left(|\mathcal{B}_H|\right)$ iterations, each of which takes time
  $\mathcal{O}\left(\ell m\right)$. Thus, the overall time complexity of these lines is $\mathcal{O}\left({3^\ell+m-1\choose m} \ell m\right)$.
\end{itemize}
\end{itemize}
\end{itemize}
Using a hashmap, the computation in line~\ref{l16} can be carried out in time proportional to the total size of the list
{\it allPossibleBivectors}, which is
$$\mathcal{O}\left(\sum_{n = 1}^{\lfloor{\ell/2\rfloor}+1}\sum_{m = n}^{\ell+n-2}{n+2\ell-1\choose 2\ell}{m+2\ell-1\choose 2\ell}(n+m)\right)\,.$$
Line~\ref{l17} takes time $\mathcal{O}(1)$.
Altogether, the time complexity of the algorithm is
$$\mathcal{O}\left(\sum_{n = 1}^{\lfloor{\ell/2\rfloor}+1}
\left({{n+1\choose 2}+\ell-1\choose \ell}
\sum_{m = n}^{\ell+n-2}
\left({n+2\ell-1\choose 2\ell}{m+2\ell-1\choose 2\ell}(n+m)+{3^\ell+m-1\choose m}\ell m\right)\right)
\right)\,.$$

Empirical measurements done while searching for the non-realizable pairs on our equipment (AMD Ryzen Threadripper 1950X 16-Core Processor) are as follows: 2,3,4 edges: neglectable, 5 edges: 5 seconds, 6 edges: 291 seconds, 7 edges: 5924 seconds. The search is stopped at this point as the next search is projected to take a few weeks.
\end{document}